\documentclass[10pt,twoside,reqno,a4paper]{amsproc}
\usepackage[utf8]{inputenc}
\usepackage{courier}
\usepackage[T1]{fontenc}

\usepackage[scaled]{helvet}
\usepackage[mathbf]{euler}
\usepackage{topformat}
\usepackage[driver=pdftex,margin=3cm,heightrounded=true,centering]{geometry}
\usepackage{hyperref}
\usepackage{enumerate}
\usepackage{amssymb}
\usepackage{amsopn}
\usepackage{amsmath}
\usepackage{mathtools}
\usepackage{tikz}
\usepackage{tikz-cd}
\usepackage{ifthen}
\usepackage{graphics}
\usepackage{amsthm}
\usepackage{topthm}
\author{Thorben Kastenholz and Mark Pedron}
\thanks{Thorben Kastenholz was supported by the SPP2026 "Geometry at Infinity"
of the DFG.}
\date{\today}
\title{The Minimal Genus of Homology Classes in a Finite 2-Complex}
\begin{document}
\newcommand{\introduce}[1]
  {\textbf{#1}}
\newcommand{\comment}[1]
  {TODO:\@{\LARGE{#1}}}

\newcommand{\apply}[2]
  {{#1}\!\left({#2}\right)}
\newcommand{\at}[2]
  {\left.{#1}\right\rvert_{#2}}
\newcommand{\absolute}[1]
  {\left\lvert{#1}\right\rvert}
\newcommand{\EuclideanNorm}[1]
  {\left\lVert{#1}\right\rVert}
\newcommand{\Cardinality}[1]
  {\left\lvert{#1}\right\rvert}
\renewcommand{\and}%
  {\wedge}
\newcommand{\NaturalNumbers}%
  {\mathbb{N}}
\newcommand{\Integers}%
  {\mathbb{Z}}
\newcommand{\Rationals}%
  {\mathbb{Q}}
\newcommand{\Reals}%
  {\mathbb{R}}

\newcommand{\AttainableSet}[2]
  {\apply{\mathcal{A}}{#2,#1}}
\newcommand{\StableAttainableSet}[2]
  {\apply{\mathcal{A}_{\text{stable}}}{#2,#1}}
\newcommand{\CompressionSmallerEqual}%
  {\leq_c}

\newcommand{\Manifold}%
  {M}
\newcommand{\Dimension}%
  {d}
\newcommand{\CollarMap}[1]
  {c_{#1}}
\newcommand{\FundamentalClass}[1]
  {\left[#1\right]}
\newcommand{\Point}%
  {\ast}
\newcommand{\Interval}%
  {I}
\newcommand{\Circle}%
  {S^{1}}
\newcommand{\Ball}[1]
  {D^{#1}}
\newcommand{\Embedding}%
  {e}
\newcommand{\GeometricDegree}%
  {\mathcal{G}}

\newcommand{\Group}%
  {G}
\newcommand{\kernel}[1]
  {\apply{\operatorname{ker}}{#1}}
\newcommand{\cokernel}[1]
  {\apply{\operatorname{coker}}{#1}}

\newcommand{\HomologyClass}%
  {\alpha}
\newcommand{\HomologyOfSpaceObject}[2]
  {\apply{H_{#2}}{#1}}
\newcommand{\HomologyOfSpaceMorphism}[2]
  {{#1}_{\ast}}
\newcommand{\HomologyOfGroupObject}[2]
  {\apply{H_{#2}}{#1}}
\newcommand{\HomologyOfGroupMorphism}[2]
  {{#1}_{\ast}}
\newcommand{\HomologyOfSpacePairObject}[3]
  {\apply{H_{#3}}{{#1},{#2}}}

\newcommand{\TopologicalSpace}%
  {X}
\newcommand{\ContinuousMap}%
  {f}
\newcommand{\PathComponentsOfSpace}[1]
  {\apply{\pi_{0}}{#1}}
\newcommand{\HomotopyGroupOfObject}[3]
  {\apply{\pi_{#3}}{{#1},{#2}}}
\newcommand{\HomotopyGroupOfPairObject}[4]
  {\apply{\pi_{#4}}{{#1},{#2},{#3}}}
\newcommand{\EMSpace}[2]
  {\apply{K}{{#1},{#2}}}

\newcommand{\Genus}%
  {g}
\newcommand{\NumberBoundaryComponents}%
  {b}
\newcommand{\OrientationIndicator}%
  {o}
\newcommand{\Surface}%
  {\Sigma}
\newcommand{\Sphere}%
  {S^{2}}
\newcommand{\OrientedSurfaceOfGB}[2]
  {\Surface_{{#1},{#2}}} %
\newcommand{\Orientable}%
  {+}
\newcommand{\NotOrientable}%
  {-}
\newcommand{\SurfaceOfGBO}[3]
  {\Surface_{{#1},{#2},{#3}}}
\newcommand{\NumberOfNonsphericalComponents}%
  {n_{0}}
\newcommand{\ChiMinus}%
  {\chi^{-}}
\newcommand{\EulerCharacteristic}%
  {\chi}
\newcommand{\WeightedEulerCharacteristic}%
  {\EulerCharacteristic_{\WeightFunction}}

\newcommand{\ManifoldModels}%
  {\mathcal{M}}
\newcommand{\ManifoldModelsOfDimension}[1]
  {\apply{\ManifoldModels}{#1}}
\newcommand{\GeneralizedComplex}%
  {X}
\newcommand{\OrientableComplex}%
  {\GeneralizedComplex_{O}}
\newcommand{\Cell}%
  {x}
\newcommand{\CellsOfDim}[2]
  {\apply{#1}{#2}}
\newcommand{\CellType}%
  {\Manifold}
\newcommand{\GluingMap}[1]
  {r_{#1}}
\newcommand{\SkeletonOfDim}[2]
  {{#1}^{({#2})}}
\newcommand{\CellInclusion}[1]
  {\phi_{#1}}
\newcommand{\GeometricRealization}[1]
  {\left\vert{#1}\right\vert}
\newcommand{\IncludedCell}[1]
  {c_{#1}}
\newcommand{\IncludedPlasma}[1]
  {p_{#1}}
\newcommand{\IncludedMembrane}[1]
  {m_{#1}}
\newcommand{\SingularSet}[1]
  {W_{#1}}
\newcommand{\SingularSetProjection}[1]
  {\pi_{#1}}
\newcommand{\Block}[1]
  {B_{#1}}
\newcommand{\BoundaryGraphVertices}[1]
  {\apply{V_{\partial}}{#1}}
\newcommand{\BoundaryGraphEdges}[1]
  {\apply{E_{\partial}}{#1}}
\newcommand{\Degree}%
  {d}
\newcommand{\DegreeInfty}%
  {\Degree_{\infty}}
\newcommand{\DegreeFinite}%
  {\Degree_{fin}}
\newcommand{\WeightFunction}%
  {w}
\newcommand{\NumberOfBoundarySegments}%
  {\operatorname{Seg}}
\newcommand{\NumberOfBoundaryCircles}%
  {\operatorname{Circ}}

\newcommand{\CATMinusOne}%
  {\operatorname{CAT}(-1)}

\begin{abstract}
  We study surface representatives of homology classes of finite complexes
  which minimize certain complexity measures, including its genus and Euler
  characteristic.
  Our main result is that up to surgery at nullhomotopic curves minimizers are
  homotopic to cellwise coverings to the 2-skeleton.
  From this we conclude that the minimizing problem is in general
  algorithmically undecidable,
  but can be solved for 2-dimensional CAT(-1)-complexes.
\end{abstract}

\maketitle
\section{Introduction}
Fix a singular integral homology class
$\HomologyClass \in \HomologyOfSpaceObject{\TopologicalSpace}{2}$
of a CW-complex $\TopologicalSpace$.
We are interested in maps of closed oriented (possibly disconnected) surfaces
$\ContinuousMap \colon \Surface \to \TopologicalSpace$
representing $\HomologyClass$,
i.e., continuous maps $f$ (not necessarily embeddings) such that
$
  \apply
    {\HomologyOfSpaceMorphism{\ContinuousMap}{2}}
    {\FundamentalClass{\Surface}}
  =
  \HomologyClass
$.
Consider the following two invariants of the representing surface
(sums are taken over the connected components of $\Surface$):
\[
  \apply{\Genus}{\Surface}
  =
  \sum_{\Surface'}
    \apply{\Genus}{\Surface'}
  \quad
  \text{and}
  \quad
  \apply{\ChiMinus}{\Surface}
  =
  \sum_{\Surface'}
    \max
      \left\{
        0, -\apply{\EulerCharacteristic}{\Surface'}
      \right\}
\]
The problem of minimizing $\Genus$ or $\ChiMinus$ among all representatives
of a given homology class has been studied under various restrictions on
$\Surface$, $\TopologicalSpace$, and $\ContinuousMap$.
If $\TopologicalSpace$ is a low-dimensional manifold,
the case where one requires $\ContinuousMap$ to be an embedding
has received a lot of attention (see below).

If $\Surface$ is connected (and not a sphere),
$\apply{\ChiMinus}{\Surface} = 2\apply{\Genus}{\Surface}-2$,
hence the two minimizing problems coincide under the restriction of
$\Surface$ being connected.
But in general,
$
  \apply{\ChiMinus}{\Surface}
  =
  2\apply{\Genus}{\Surface}
  -
  2\apply{\NumberOfNonsphericalComponents}{\Surface}
$,
where $\apply{\NumberOfNonsphericalComponents}{\Surface}$ is the
number of non-spherical components of $\Surface$.
Hence the problems can (and will) differ.
To capture this difference, we define the
\introduce{attainable set} $\AttainableSet{\TopologicalSpace}{\HomologyClass}$
as
\[
  \AttainableSet{\TopologicalSpace}{\HomologyClass}
  =
  \left\{
    \left(
      \apply{\ChiMinus}{\Surface}, \apply{\Genus}{\Surface}
    \right)
  \middle\vert
    \ContinuousMap \colon \Surface \to \TopologicalSpace,
    \apply
      {\HomologyOfSpaceMorphism{\ContinuousMap}{2}}
      {\FundamentalClass{\Surface}}
    =
    \HomologyClass
  \right\}.
\]
Let us from now on write $\apply{\ChiMinus}{\HomologyClass}$ and
$\apply{\Genus}{\HomologyClass}$ for the minimum of these invariants
over all the surface representatives of $\HomologyClass$.

\subsection{Results}
In Theorem~\ref{thm:EffectiveComputability}, we provide an algorithm which
computes the attainable set for some class of finite $2$-complexes
$\TopologicalSpace$, which includes all CAT$(-1)$-2-complexes.
Conversely we show that for general finite $2$-complexes one cannot compute
$\Genus$ and $\ChiMinus$.
\begin{theorem}[Minimal Genus is Undecidable]%
\label{thm:MinimalGenusIsUndecidable}
  Let $B$ be a positive integer, and $i$ be either $\ChiMinus$ or $\Genus$.
  Then there is no algorithm, taking as input an
  (encoding of a) finite 2-complex $\GeneralizedComplex$ and a homology class
  $
    \HomologyClass
    \in
    \HomologyOfSpaceObject{\GeometricRealization{\GeneralizedComplex}}{2}
  $,
  which outputs whether $i\left(\omega\right)\leq B$.
\end{theorem}

The restriction to $2$-complexes is inessential
due to the fact that the attainable set of $\HomologyClass$
depends only on the image of $\HomologyClass$ in the homology
of the fundamental group of $\TopologicalSpace$
(Corollary~\ref{crl:AttainableSetFundamentalGroup}).

A crucial ingredient for our investigation of representatives of a homology
class is an extension of a theorem by Edmonds
(see \cite{EdmondsBranchedCoverings}, and \cite{SkoraDegree} for a
strengthened result) in Theorem~\ref{thm:NormalForm}.
This extension can be found at the end of the introduction.
It yields the following corollary.
\begin{corollary}
\label{crl:CellwiseCovering}
  Given a combinatorial generalized $2$-complex $\GeneralizedComplex$ and a
  second homology
  class $\alpha\in \HomologyOfSpaceObject{\GeneralizedComplex}{2}$, for
  every representative $\left(\Surface,\ContinuousMap\right)$ there exists a
  new representative $\left(\Surface',\ContinuousMap'\right)$ such that
  $\apply{\ChiMinus}{\Surface'}\leq\apply{\ChiMinus}{\Surface}$,
  $\apply{\Genus}{\Surface'}\leq\apply{\Genus}{\Surface}$, and
  $\ContinuousMap'$ is a cellwise covering without folds.
\end{corollary}

By a combinatorial generalized $2$-complex we mean a $2$-complex whose $2$-cells
can be arbitrary compact, connected surfaces with boundary such that the gluing
maps are of a particular form.
Up to homotopy equivalence this yields the same topological spaces as ordinary
$2$-complexes.
Having no folds means that the map is locally injective outside of a finite set
contained in the preimage of the zero cells.

In Section~\ref{scn:AttainableSet} we prove basic results about the structure
of attainable sets. In Section~\ref{scn:Hurewicz} we relate the attainable sets
of a space to its fundamental group and in particular show that the attainable
sets of a space can be retrieved from the attainable sets of its $2$-skeleton.
In Section~\ref{scn:TwoComplexes} we will introduce generalized $2$-complexes
in order to prove Theorem~\ref{thm:NormalForm} in
Section~\ref{scn:CellwiseCovering}.
Theorem~\ref{thm:EffectiveComputability} will be proven in
Section~\ref{scn:WeightFunctions}. Lastly in Section~\ref{scn:Undecidability}
we prove Theorem~\ref{thm:MinimalGenusIsUndecidable}.
\subsection{Background}
For three-dimensional compact oriented manifolds Thurston (\cite{ThurstonNorm})
introduced a pseudo-norm on their second homology:
Let $\Manifold$ denote such a $3$-manifold, then
$
  \EuclideanNorm{-}_{\textrm{T}}
  \colon
  \HomologyOfSpaceObject{\Manifold}{2}
  \to
  \Integers
$
assigns to every homology class the minimum of
$\ChiMinus$ over all \emph{embedded} representatives of that class.
In particular, he showed that $\EuclideanNorm{-}_{\textrm{T}}$
is multiplicative, i.e.\ %
$
  \EuclideanNorm{n\HomologyClass}_{\textrm{T}}%
  =
  n\EuclideanNorm{\HomologyClass}_{\textrm{T}}%
$.
For non-primitive classes $\HomologyClass$ these representatives are forced to
be disconnected.
Gromov (\cite{GromovBoundedCohomology}) introduced an a priori unrelated
pseudo-norm $\EuclideanNorm{-}_{\textrm{G}}$ on homology with real coefficients
by considering the infimum of the $\ell^{1}$-norm of real cycles representing
the homology class.
He showed that this relates to the stabilization of $\ChiMinus$ via
$
  \EuclideanNorm{\HomologyClass}_{\textrm{G}}%
   =
   2\lim_{n\to\infty}
   \frac{\apply{\ChiMinus}{n\HomologyClass}}{n}%
$.
Using Thurston's results and the theory of taut foliations,
Gabai (\cite{GabaiFoliations}, Corollary~6.18) showed that
$
  \EuclideanNorm{\HomologyClass}_{\textrm{G}}%
  =
  2 \EuclideanNorm{\HomologyClass}_{\textrm{T}}%
$.
We can arrange these results in the following chain of inequalities:
\[
  \EuclideanNorm{\HomologyClass}_{\textrm{T}}
  \geq
  \apply{\ChiMinus}{\HomologyClass}
  \geq
  \frac{\apply{\ChiMinus}{n\HomologyClass}}{n}
  \stackrel{\textrm{Gromov}}{\to}
  \frac{1}{2}\EuclideanNorm{\HomologyClass}_{\textrm{G}}
  \stackrel{\textrm{Gabai}}{=}
  \EuclideanNorm{\HomologyClass}_{\textrm{T}}
\]
Hence $\apply{\ChiMinus}{\HomologyClass}$,
$\EuclideanNorm{\HomologyClass}_{\textrm{T}}$,
and $\frac{1}{2} \EuclideanNorm{\HomologyClass}_{\textrm{G}}$
are equal,
thus the Thurston norm can be defined using arbitrary (non-embedded) surfaces,
\emph{even without stabilization}.
This result clearly distinguishes $\ChiMinus$ from $\Genus$.
In general, even though
$
\lim_{n\to\infty}
\frac{\apply{\ChiMinus}{n\HomologyClass}}{n}
=
2 \lim_{n\to\infty} \frac{\apply{\Genus}{n\HomologyClass}}{n}
$
(this is true for arbitrary topological spaces),
$\Genus$ is not minimized by an embedded surface.
Additionally, while the equality above shows that $\ChiMinus$ is
multiplicative on homology classes,
$\Genus$ (and even $(\Genus-1)$) does not enjoy such multiplicative properties.
Let us illustrate this with an example:
\begin{example}[Genus 2 Handlebody With Three Tori Removed]%
\label{exm:Handlebody}
  \begin{figure}[h]
    \def\svgwidth{0.5\textwidth}
\begingroup%
  \makeatletter%
  \providecommand\color[2][]{%
    \errmessage{(Inkscape) Color is used for the text in Inkscape, but the package 'color.sty' is not loaded}%
    \renewcommand\color[2][]{}%
  }%
  \providecommand\transparent[1]{%
    \errmessage{(Inkscape) Transparency is used (non-zero) for the text in Inkscape, but the package 'transparent.sty' is not loaded}%
    \renewcommand\transparent[1]{}%
  }%
  \providecommand\rotatebox[2]{#2}%
  \newcommand*\fsize{\dimexpr\f@size pt\relax}%
  \newcommand*\lineheight[1]{\fontsize{\fsize}{#1\fsize}\selectfont}%
  \ifx\svgwidth\undefined%
    \setlength{\unitlength}{510.75002358bp}%
    \ifx\svgscale\undefined%
      \relax%
    \else%
      \setlength{\unitlength}{\unitlength * \real{\svgscale}}%
    \fi%
  \else%
    \setlength{\unitlength}{\svgwidth}%
  \fi%
  \global\let\svgwidth\undefined%
  \global\let\svgscale\undefined%
  \makeatother%
  \begin{picture}(1,0.506916)%
    \lineheight{1}%
    \setlength\tabcolsep{0pt}%
    \put(0,0){\includegraphics[width=\unitlength,page=1]{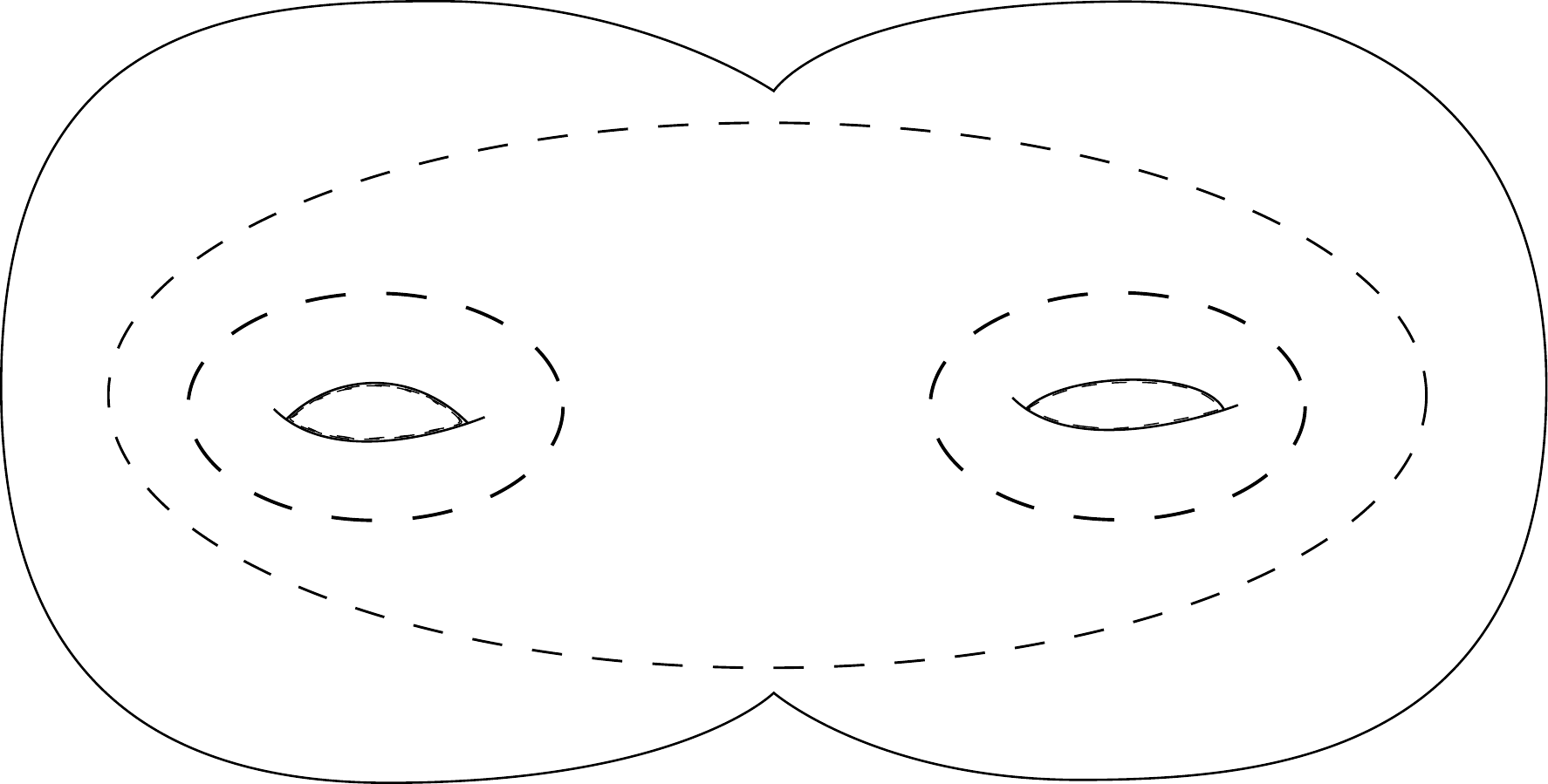}}%
    \put(0.32691349,0.28784221){\color[rgb]{0,0,0}\makebox(0,0)[lt]{\lineheight{1.25}\smash{\begin{tabular}[t]{l}$a$\end{tabular}}}}%
    \put(0.61332704,0.30882942){\color[rgb]{0,0,0}\makebox(0,0)[lt]{\lineheight{1.25}\smash{\begin{tabular}[t]{l}$b$\end{tabular}}}}%
    \put(0.56456264,0.42919711){\color[rgb]{0,0,0}\makebox(0,0)[lt]{\lineheight{1.25}\smash{\begin{tabular}[t]{l}$ab$\end{tabular}}}}%
  \end{picture}%
\endgroup%

    \caption{%
      The space $\TopologicalSpace$ in Example~\ref{exm:Handlebody}
    }
    \label{fig:Handlebody}
  \end{figure}
  Let $\TopologicalSpace$ be the link complement in a genus $2$
  handlebody as depicted in Figure~\ref{fig:Handlebody}.
  If $a,b$ denote the generators of the fundamental group of the handlebody,
  the $3$ link components represent the conjugacy classes of $a$, $b$ and
  $ab$.
  Each of the components is surrounded by an embedded torus.
  The fundamental classes of these tori are a basis of
  $\HomologyOfSpaceObject{\TopologicalSpace}{2}$.
  The space $\TopologicalSpace$ is homotopy equivalent to a $2$-complex
  given by gluing $3$ squares as depicted in Figure~\ref{fig:3Squares} and
  identifying opposite sides.
  The classes given by the $3$ squares correspond to these
  fundamental classes.
  Let $\HomologyClass$ be the sum of the $3$ squares with the orientation
  induced by a fixed orientation of the surrounding plane in
  Figure~\ref{fig:3Squares}.
  \begin{figure}[h]
    \def\svgwidth{0.5\textwidth}
    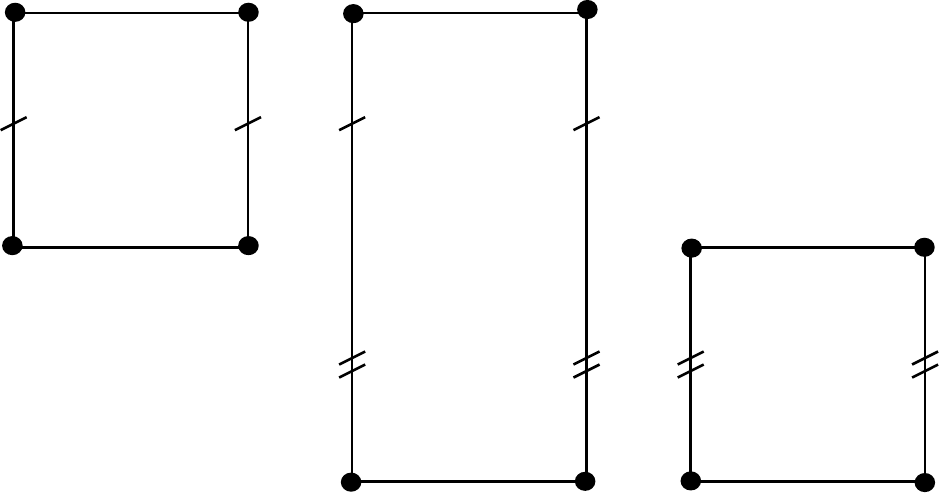
    \caption{%
      The Gluing Pattern of $\TopologicalSpace$ in Example~\ref{exm:Handlebody}
    }
    \label{fig:3Squares}
  \end{figure}
  There are two obvious surface representatives of $\HomologyClass$ one
  is given by $3$ tori, the other one is a genus $2$ surface
  bounding the handlebody.
  Using Corollary~\ref{crl:CellwiseCovering},
  an elementary but tedious consideration of small combinations
  of the squares shows that $\HomologyClass$ can not be represented by
  less than $3$ tori.
  Hence $\AttainableSet{\TopologicalSpace}{\HomologyClass}$ looks as
  depicted in Figure~\ref{fig:AttainableSetHandlebody}.
  \begin{figure}[h]
      \begin{tikzpicture}
    \clip (-1,-0.5) rectangle (8.8,6.7);
    \filldraw[white!80!black] (0,11) -- (0,3) -- (2,3) -- (2,2)
    --(20,11)--cycle;
    \def\crosslength{0.05} 
    \foreach \y in {3,4,...,10} {
      \pgfmathsetmacro{\fofy}{2*\y-2}
      \foreach \x in {0,2,...,10}{
          \ifthenelse{\x < \fofy \OR \x=\fofy} 
            {\draw (\x+\crosslength,\y+\crosslength) --
                +(-2*\crosslength,-2*\crosslength);
                \draw (\x-\crosslength,\y+\crosslength) --
                +(2*\crosslength,-2*\crosslength);
            }{}
        }
    };
    \draw (2+\crosslength,2+\crosslength) --
    +(-2*\crosslength,-2*\crosslength);
    \draw (2-\crosslength,2+\crosslength) --
    +(2*\crosslength,-2*\crosslength);
    \draw[dashed] (0,1) -- (11,6.5);
    \def\dashlength{0.05}
    \draw[->] (0,0) --  (8.7,0) node[anchor= north east]{$\chi^-$};
    \foreach \x in {1,2,...,8}{
        \draw (\x,-\dashlength) -- (\x,0) node[below]{$\x$} -- (\x,\dashlength);
    }
    \node at (-0.3,6.4) {$g$};
    \draw[->] (0,0) -- (0,6.7);
    \foreach \y in {1,2}
    {
        \draw (-\dashlength,\y)-- (0,\y) node[left]{$\y$} -- (\dashlength, \y);
    };
    \foreach \y in {2,...,9}{
        \node[left] at (0,\y) {$\y$};
    }
    \draw (-\dashlength,-\dashlength)-- (0,0) node[below left]{$0$} --
    (\dashlength,\dashlength);
  \end{tikzpicture}
    \caption{%
      The Attainable Set of $\TopologicalSpace$ in Example~\ref{exm:Handlebody}
    }
    \label{fig:AttainableSetHandlebody}
  \end{figure}
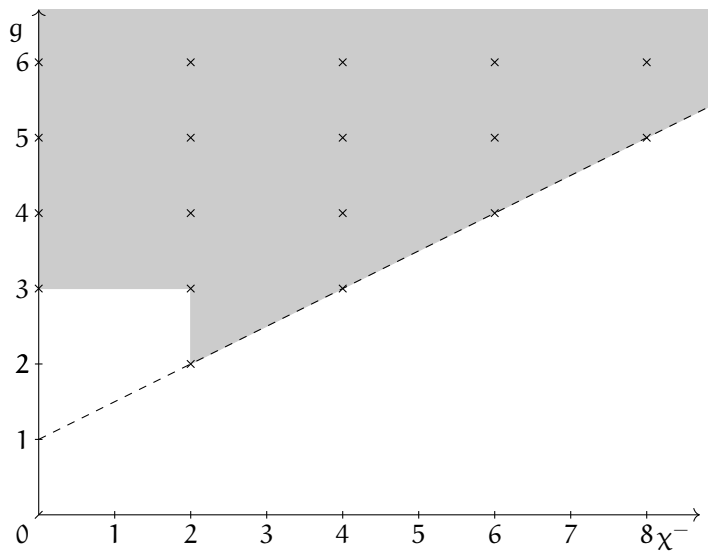
  In particular, there is no representative simultaneously minimizing
  $\ChiMinus$ and $\Genus$.
  Because $2\HomologyClass$ is represented by $3$	tori as well, we see
  that $\apply{\Genus}{2\HomologyClass}=3\neq
  4=2\apply{\Genus}{\HomologyClass}$, hence $\Genus$ is not multiplicative.
\end{example}

For four-dimensional manifolds the situation is different, as in this case
the restriction to embedded surfaces can influence the minimal genus
drastically.
The Thom conjecture states that given an embedded smooth projective curve
with genus $\Genus$ in $\mathbb{C}P^2$ (which is automatically connected),
every other embedded representative of the same homology class has
genus greater or equal than $\Genus$.
Using Seiberg-Witten invariants, this was proved by Kronheimer and Mrowka
(\cite{KronheimerMrowkaThomConjecture}).
Note that since $\mathbb{C}P^2$ is simply-connected,
every second homology class can be represented by a non-embedded sphere.
Their result was further strengthened by Ozsv\'{a}th and Szab\'{o}
(\cite{OszvathSzaboThomConjecture}),
again using Seiberg-Witten invariants.
They showed that every connected symplectic subsurface
of a symplectic four-manifold has minimal genus
among all embedded representatives of the same homology class.

There is also a relative analogue of the minimal genus of a homology class,
defined via extensions of maps
$\ContinuousMap \colon \Circle \to \TopologicalSpace$
to a connected surface with one boundary component.
The minimal genus of such an extension is the commutator length
of the conjugacy class of the element in
$\HomotopyGroupOfObject{\TopologicalSpace}{x}{1}$
corresponding to $\ContinuousMap \colon \Circle \to \TopologicalSpace$.
One can also stabilize the commutator length to obtain
the stable commutator length $\text{scl}$.
See~\cite{CalegariScl} for an introduction.
Some of our methods are similar to approaches in that context
(in particular~\cite{ChenSclGraphGroups},~\cite{CalegariSclRational},
and~\cite{CalegariSurfaceSubgroups}).

Similar to how commutator length carries more information than the stable
commutator length, the attainable set loses most of its structure when
stabilized (Lemma~\ref{lem:AttainableSetStabilization}).
Let us illustrate this with an example:
\begin{example}[Three Octagons]
  \label{exm:ThreeOctagons}
  The necessary calculations for this example are deferred to
  Example~\ref{exm:ThreeOctagonsContinued}.
  Consider the $2$-complex $\TopologicalSpace$ given by the gluing pattern in
  Figure~\ref{fig:GluingThreeOctagons} and identifying opposite sides.
  Each of the three cells is a closed surface of genus $2$, hence
  its homology has rank $3$ and is generated by its cells.
  We consider the homology class $\HomologyClass$ which is given as the sum of
  the three cells.
  The attainable sets of $\HomologyClass$ and its multiples are depicted in
  the left diagram of Figure~\ref{fig:AttainableSetThreeOctagons}.
  In this example $\ChiMinus$ and $\Genus$ are never minimized by the same
  representative.
  Moreover, all $\Genus$-minimizers are connected (up to sphere components).
  In the right diagram, the attainable sets
  $\AttainableSet{\TopologicalSpace}{n\HomologyClass}$ is rescaled, dividing by $n$.
  One can see how the subtle structure of the attainable set vanishes under
  stabilization.
  \begin{figure}[h]%
    \def\svgwidth{0.7\textwidth}
    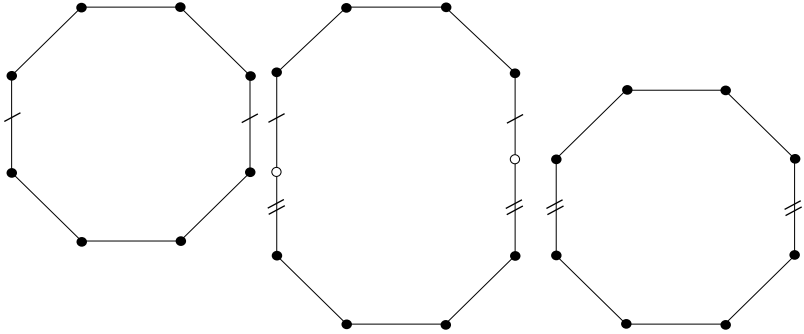
    \caption{%
      Gluing three octagons, Example~\ref{exm:ThreeOctagons}%
      \label{fig:GluingThreeOctagons}
    }
  \end{figure}
  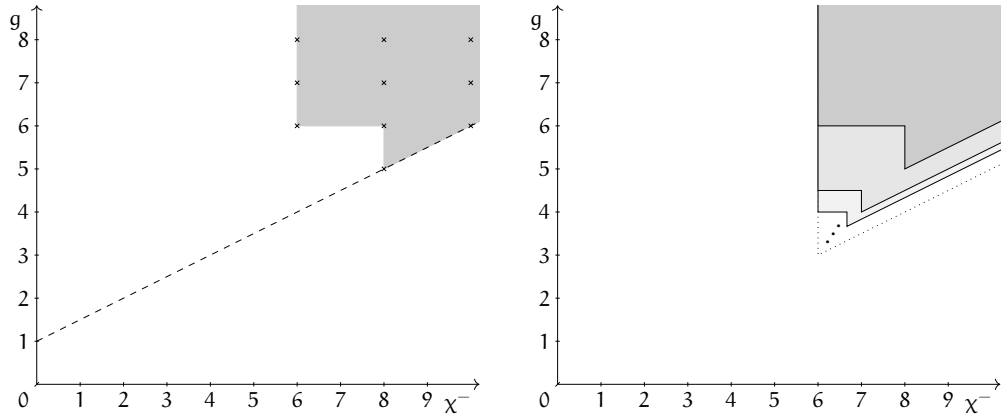
\begin{figure}[h]%
    \resizebox{.9\textwidth}{!}{
      \begin{tikzpicture}[scale=0.7]
\begin{scope}
    \clip (-1,-1) rectangle (10.2,8.8);
    \filldraw[white!80!black] (6,11) -- (6,6) -- (8,6) --
    (8,5)--(20,11)--cycle;
    \def\crosslength{0.05} 
    \foreach \y in {6,7,...,10} {
        \pgfmathsetmacro{\fofy}{2*\y-2}
        \foreach \x in {6,8,10}{
            \ifthenelse{\x < \fofy \OR \x=\fofy} 
            {\draw (\x+\crosslength,\y+\crosslength) --
                +(-2*\crosslength,-2*\crosslength);
                \draw (\x-\crosslength,\y+\crosslength) --
                +(2*\crosslength,-2*\crosslength);
            }{}
        }
    };
    \draw (8+\crosslength,5+\crosslength) --
    +(-2*\crosslength,-2*\crosslength);
    \draw (8-\crosslength,5+\crosslength) --
    +(2*\crosslength,-2*\crosslength);
    \draw[dashed] (0,1) -- (11,6.5);
    \def\dashlength{0.05}
    \draw[->] (0,0) --  (10.2,0) node[anchor= north
    east]{$\chi^-$};
    \foreach \x in {1,2,...,9}{
        \draw (\x,-\dashlength) -- (\x,0) node[below]{$\x$}
        -- (\x,\dashlength);
    }
    \node at (-0.5,8.4) {$g$};
    \draw[->] (0,0) -- (0,8.8);
    \foreach \y in {1,2,...,8}
    {
        \draw (-\dashlength,\y)-- (0,\y) node[left]{$\y$} --
        (\dashlength, \y);
    };
    \draw (-\dashlength,-\dashlength)-- (0,0) node[below
    left]{$0$} --
    (\dashlength,\dashlength);
\end{scope}
\begin{scope}[xshift=12cm]
    \clip (-1,-1) rectangle (10.2,8.8);
    \draw[dotted] (6,11) -- (6,3) -- (20,10) --cycle;
    \filldraw[white!95!black] (6,11) -- (6,4) -- (6.6666666,4) --
    (6.6666666,3.6666666)--(20,10.333333)--cycle;
    \draw (6,11) -- (6,4) -- (6.6666666,4) --
    (6.6666666,3.6666666)--(20,10.333333)--cycle;
    \filldraw[white!90!black] (6,11) -- (6,4.5) -- (7,4.5) --
    (7,4)--(20,10.5)--cycle;
    \draw (6,11) -- (6,4.5) -- (7,4.5) --
    (7,4)--(20,10.5)--cycle;
    \filldraw[white!80!black] (6,11) -- (6,6) -- (8,6) --
    (8,5)--(20,11)--cycle;
    \draw (6,11) -- (6,6) -- (8,6) --
    (8,5)--(20,11)--cycle;
    \node at (6.35,3.5) [rotate=55]{$\ldots$};
    \def\dashlength{0.05}
    \draw[->] (0,0) --  (10.2,0) node[anchor= north
    east]{$\chi^-$};
    \foreach \x in {1,2,...,9}{
        \draw (\x,-\dashlength) -- (\x,0) node[below]{$\x$}
        -- (\x,\dashlength);
    }
    \node at (-0.5,8.4) {$g$};
    \draw[->] (0,0) -- (0,8.8);
    \foreach \y in {1,2,...,8}
    {
        \draw (-\dashlength,\y)-- (0,\y) node[left]{$\y$} --
        (\dashlength, \y);
    };
    \draw (-\dashlength,-\dashlength)-- (0,0) node[below
    left]{$0$} --
    (\dashlength,\dashlength);
\end{scope}
\end{tikzpicture}
    }
    \caption{%
      Attainable sets in Example~\ref{exm:ThreeOctagons}%
      \label{fig:AttainableSetThreeOctagons}
    }
  \end{figure}
\end{example}

\subsection{Compression Order}
In order to get a better understanding of the attainable set,
we will consider a preorder%
\footnote{%
  i.e.,\ %
  $A \CompressionSmallerEqual A$, and
  $
     A \CompressionSmallerEqual B
     \and
     B \CompressionSmallerEqual C
     \implies
     A \CompressionSmallerEqual C
  $,
  but not necessarily
  $
     A \CompressionSmallerEqual B
     \and
     B \CompressionSmallerEqual A
     \implies
     A = B
  $
}
$\CompressionSmallerEqual$ on (unparameterized homotopy classes of)
representatives of a homology class, such that $\Genus$ and $\ChiMinus$
are monotonous with respect to this preorder.
This is based on the following construction:
\begin{definition}[Compression]
  For a given representative $(\Surface,\ContinuousMap)$
  of a second homology class
  $
    \HomologyClass
    \in
    \HomologyOfSpaceObject{\TopologicalSpace}{2}
  $
  with an embedded annulus
  $e \colon \Circle \times \Ball{1} \to \Surface$
  and an extension
  \[
    \begin{tikzcd}
      \Circle \times \Ball{1} \ar[d,hook]\ar[r,"\Embedding"]
        & \Surface \ar[d,"\ContinuousMap"]\\
      \Ball{2} \times \Ball{1} \ar[r,"\Embedding'"]
        & \TopologicalSpace
    \end{tikzcd}
  \]
  we set
  \begin{itemize}
    \item
      $
        \Surface_{\Embedding}
        =
        \Surface
        \setminus
        \mathring{
          \apply{\Embedding}{\Circle \times \Ball{1}}
        }
        \cup
        \Ball{2} \times \partial \Ball{1}
      $
    \item
      $
        \ContinuousMap_{\Embedding'}
        \colon
        \Surface_{\Embedding}
        \to
        \TopologicalSpace
        =
        \at
          {\ContinuousMap}
          {
            \Surface
            \setminus
            \apply{\Embedding}{\Circle \times \Ball{1}}
          }
        \cup
        \at{\Embedding'}{\Ball{2} \times \partial \Ball{1}}
      $
  \end{itemize}
  We will say that
  $
    \left(
      \Surface_{\Embedding}
      ,
      \ContinuousMap_{\Embedding'}
    \right)
  $
  is \introduce{obtained from $(\Surface,\ContinuousMap)$
  by compression at $\Embedding$}.
  By construction,
  $
    \left(
      \Surface_{\Embedding}
      ,
      \ContinuousMap_{\Embedding'}
    \right)
  $
  represents the same homology class as
  $
    \left(
      \Surface
      ,
      \ContinuousMap
    \right)
  $.
\end{definition}

Consider maps
$\ContinuousMap_{i} \colon \Surface_{i} \to \TopologicalSpace$,
$i \in \left\{ 1, 2 \right\}$,
representing the fixed homology class $\HomologyClass$.
We define the \introduce{compression preorder} by saying
that
$
  (\Surface_{1}, \ContinuousMap_{1})
  \CompressionSmallerEqual
  (\Surface_{2}, \ContinuousMap_{2})
$
if $(\Surface_{1}, \ContinuousMap_{1})$ can be obtained from
$(\Surface_{2}, \ContinuousMap_{2})$
by a finite sequence of the following moves:
\begin{enumerate}[(I)]
  \item%
  \label{itm:EssentialCompression}
    Compression at an essential embedded annulus
  \item%
  \label{itm:InessentialCompression}
    Compression at an inessential embedded annulus
  \item%
  \label{itm:InessentialInflation}
    The inverse of move~(\ref{itm:InessentialCompression})
\end{enumerate}
It is necessary to allow the inverse of move~(\ref{itm:InessentialCompression})
as otherwise the preorder would not have minima%
\footnote{%
  A minimum $A$ of a preorder is an element such that for all elements $B$,
  $B \leq A \implies A \leq B$.
}%
.
The compression order gives the universal measurement of geometric complexity
of surface representatives of a given homology class.
The important role of $\left(\ChiMinus,\Genus\right)$ comes from the fact that
the assignment
\[
  \left(\Surface,\ContinuousMap\right)
  \mapsto
  \left(
    \apply{\ChiMinus}{\Surface},
    \apply{\Genus}{\Surface}
  \right)
\]
is strictly monotone%
\footnote{%
  $%
    \left(\Surface_{1},\ContinuousMap_{1}\right)
    \CompressionSmallerEqual
    \left(\Surface_{2},\ContinuousMap_{2}\right)
    \Rightarrow
    \left(
      \apply{\ChiMinus}{\Surface_{1}},
      \apply{\Genus}{\Surface_{1}}
    \right)
    \leq
    \left(
      \apply{\ChiMinus}{\Surface_{2}},
      \apply{\Genus}{\Surface_{2}}
    \right)
  $
  and
  $%
    \left(\Surface_{1},\ContinuousMap_{1}\right)
    \CompressionSmallerEqual
    \left(\Surface_{2},\ContinuousMap_{2}\right)
    \and
    \left(\Surface_{2},\ContinuousMap_{2}\right)
    \not\CompressionSmallerEqual
    \left(\Surface_{1},\ContinuousMap_{1}\right)
    \Rightarrow
    \hphantom{.........}
    \left(
      \apply{\ChiMinus}{\Surface_{2}},
      \apply{\Genus}{\Surface_{2}}
    \right)
    \not\leq
    \left(
      \apply{\ChiMinus}{\Surface_{1}},
      \apply{\Genus}{\Surface_{1}}
    \right)
  $
}
from the compression order to the product order on the attainable set.
In particular, $\left(\ChiMinus,\Genus\right)$ reflects minima%
\footnote{%
  Reflects minima:
  $
    \left(
      \apply{\ChiMinus}{\Surface},
      \apply{\Genus}{\Surface}
    \right)
    \Rightarrow
    \left(\Surface,\ContinuousMap\right)
    \textrm{ minimum}
  $
}.

This allows us to formulate Theorem~\ref{thm:NormalForm} which
describes minima of the compression preorder up to equivalence in the preorder%
\footnote{%
  $A$ and $B$ are called equivalent if
  $A \CompressionSmallerEqual B \and B \CompressionSmallerEqual A$
}%
of generalized $2$-complexes:
\begin{theorem}[Normal Form]
\label{thm:NormalForm}
  Let $\TopologicalSpace$ denote a combinatorial generalized $2$-complex and
  $\Surface$ a closed oriented surface.
  Then every map $\ContinuousMap\colon \Surface \to \TopologicalSpace$ that
  admits no squeezes is equivalent in the compression order to a cellwise
  covering without folds.
\end{theorem}

\begin{acknowledgement}
  The authors would like to thank Ursula Hamenstädt for her support and many
  helpful discussions.
  We would like to thank the current and former members
  of the Bonn Geometry Group for introducing us to new topics and
  sharing their insights with us.
\end{acknowledgement}

\section{Attainable Sets}
\label{scn:AttainableSet}
We continue to denote by $\TopologicalSpace$ a CW-complex and by
$\HomologyClass \in \HomologyOfSpaceObject{\TopologicalSpace}{2}$
an integral homology class.
As described in the introduction,
the invariants $\apply{\ChiMinus}{\HomologyClass}$
and
$\apply{\Genus}{\HomologyClass}$
are neither independent nor does one determine the other.
To capture their interaction we defined the attainable set
$\AttainableSet{\TopologicalSpace}{\HomologyClass}$ as
\[
  \AttainableSet{\TopologicalSpace}{\HomologyClass}
  =
  \left\{
    \left(
      \apply{\ChiMinus}{\Surface},
      \apply{\Genus}{\Surface}
    \right)
  \middle\vert
    \ContinuousMap \colon \Surface \to \TopologicalSpace,
    \apply
      {\HomologyOfSpaceMorphism{\ContinuousMap}{2}}
      {\FundamentalClass{\Surface}}
    =
    \HomologyClass
  \right\}.
\]
Its geometry describes the way those two invariants restrain each other.
In this section we present results regarding its asymptotic structure.
If $\HomologyClass$ is representable by a union of spheres,
its attainable set is
\[
  \AttainableSet{\TopologicalSpace}{\HomologyClass}
  =
  \AttainableSet{\TopologicalSpace}{0}
  =
  \left\{
    \left( \ChiMinus , \Genus \right)
    \in
    2\Integers \times \Integers
  \middle\vert
    0 \leq \ChiMinus \leq 2 \Genus - 2
  \right\}
  \cup
  \left\{
    \left( 0 , 0 \right)
  \right\}.
\]
This attainable set is an exception in many ways,
so from now on in this section
\emph{%
  we assume that $\HomologyClass$ is not representable by a union of spheres%
},
i.e., $\apply{\Genus}{\HomologyClass} > 0$.
Be aware, though, that this does not exclude multiples of $\HomologyClass$ to
be representable by a union of spheres.

\subsection{Saturation}
By the inequalities
$
  0 \leq \apply{\ChiMinus}{\Surface} \leq 2\apply{\Genus}{\Surface}-2
$
for a surface $\Surface$ representing $\HomologyClass$,
the attainable set is contained in a cone whose edges are parallel to
the lines defined by $\ChiMinus=0$ and $\ChiMinus=2\Genus$.
We show that the attainable set is invariant under
saturation moves parallel to these axes.
\begin{lemma}[Invariance under Saturation Moves]%
\label{lem:AttainableSetSaturation}
  For $\HomologyClass \in \HomologyOfSpaceObject{\TopologicalSpace}{2}$
  with $\apply{\Genus}{\HomologyClass} > 0$, one has
  \[
    \left( 0 , 0 \right)
    \neq
    \left( \ChiMinus , \Genus \right)
    \in
    \AttainableSet{\TopologicalSpace}{\HomologyClass}
    \implies
    \left( \ChiMinus , \Genus + 1\right),
    \left( \ChiMinus + 2, \Genus + 1\right)
    \in
    \AttainableSet{\TopologicalSpace}{\HomologyClass}.
  \]
\end{lemma}
\begin{proof}
  Let $\Surface$ be a surface representing $\HomologyClass$ with
  $
    \left(
      \apply{\ChiMinus}{\Surface},
      \apply{\Genus}{\Surface}
    \right)
    =
    \left( \ChiMinus , \Genus \right)
  $.
  Taking the connected or disconnected sum with a torus at a
  non-spherical component changes the coordinates as required.
\end{proof}

Hence $\AttainableSet{\TopologicalSpace}{\HomologyClass}$ will agree,
up to a finite set, with the set
\begin{equation}%
\label{eqn:Cone}
  \left\{
    \left(
      \ChiMinus, \Genus
    \right)
    \in
    2\Integers \times \Integers
  \middle\vert
    \apply{\ChiMinus}{\HomologyClass}
    \leq
    \ChiMinus
    \leq
    2\Genus-2\apply{A}{\HomologyClass}
  \right\},
\end{equation}
\begin{equation*}
  \apply{A}{\HomologyClass}
  =
  \min
  \left\{
    2\Genus - \ChiMinus
  \middle\vert
    \left( \Genus , \ChiMinus \right)
    \in
    \AttainableSet{\TopologicalSpace}{\HomologyClass}
  \right\}.
\end{equation*}
We will identify the meaning of this number below.

\subsection{Connected Representative}
Let us call a surface $\Surface$ essentially connected if it has exactly one
non-spherical component.
If $\Surface$ is an essentially connected surface representing
$\HomologyClass$,
one necessarily has
$\apply{\ChiMinus}{\Surface}=2\apply{\Genus}{\Surface}-2$.
Conversely, if
$
  \left( \ChiMinus , \Genus \right)
  \in
  \AttainableSet{\TopologicalSpace}{\HomologyClass}
$
satisfies $\ChiMinus=2\Genus-2$,
one can find an essentially connected surface
$\Surface$ representing $\HomologyClass$ such that
$
  \left(
    \apply{\ChiMinus}{\Surface},
    \apply{\Genus}{\Surface}
  \right)
  =
  \left(\ChiMinus,\Genus\right).
$
Hence $\HomologyClass$ is represented by an essentially connected surface
if and only if $\apply{A}{\HomologyClass}=1$.
More generally, we have the following lemma:
\begin{lemma}%
\label{lem:EssentialSupport}
  Let
  $\TopologicalSpace = \bigsqcup \TopologicalSpace_{i}$
  be the decomposition of
  $\TopologicalSpace$
  into its path-components and let
  $
    \HomologyClass
    =
    {\left(
      \HomologyClass_{i}
    \right)}_{i}
    \in
    \bigoplus
    \HomologyOfSpaceObject{\TopologicalSpace_{i}}{2}
    \cong
    \HomologyOfSpaceObject{\TopologicalSpace}{2}
  $
  be a homology class of $\TopologicalSpace$.
  Then $\apply{A}{\HomologyClass}$ equals the number of $i$
  such that $\HomologyClass_{i}$ is not representable by a sphere.
\end{lemma}
\begin{proof}
  If $\Surface$ is a surface representing $\HomologyClass$ having
  $\apply{\NumberOfNonsphericalComponents}{\Surface}$
  non-spherical components, then
  $
    2\apply{\Genus}{\Surface}
    -
    \apply{\ChiMinus}{\Surface}
    =
    2\apply{\NumberOfNonsphericalComponents}{\Surface}.
  $
  Taking the minimum of both sides over all $\Surface$ representing
  $\HomologyClass$ proves the assertion.
\end{proof}

Thus we can write
$\apply{\NumberOfNonsphericalComponents}{\HomologyClass}$
instead of
$\apply{A}{\HomologyClass}$.
The coordinates of all surfaces $\Surface$ representing $\HomologyClass$ with
$
  \apply{\NumberOfNonsphericalComponents}{\Surface}
  =
  \apply{\NumberOfNonsphericalComponents}{\HomologyClass}
$
lie on a positive ray.  Hence there is a minimal point on this ray.
Let us write
$
  \left(
    \apply{\ChiMinus_{c}}{\HomologyClass},
    \apply{\Genus_{c}}{\HomologyClass}
  \right)
$
for its coordinates.
\begin{proposition}[Attainable Set Bounds]%
\label{prp:AttainableSetBounds}
  Consider $\HomologyClass \in \HomologyOfSpaceObject{\TopologicalSpace}{2}$
  with $\apply{\Genus}{\HomologyClass} > 0$.
  \begin{enumerate}[(a)]
    \item\label{itm:ConditionalSaturation}
      The attainable set is invariant under the following (conditional)
      saturation move:
      \[
        \left(\ChiMinus,\Genus\right)
        \in
        \AttainableSet{\TopologicalSpace}{\HomologyClass}
        \and
        \ChiMinus
        <
        2 \Genus - 2 \apply{\NumberOfNonsphericalComponents}{\HomologyClass}
        \implies
        \left(\ChiMinus+2,\Genus\right)
        \in
        \AttainableSet{\TopologicalSpace}{\HomologyClass}
      \]
    \item\label{itm:Bounds}
      $
        \AttainableSet{\TopologicalSpace}{\HomologyClass}
        \subseteq
        \left\{
          \left( \ChiMinus , \Genus \right)
        \middle\vert
            \apply{\Genus_{c}}{\HomologyClass}
            \leq
            \Genus
          \and
            \apply{\ChiMinus}{\HomologyClass}
            \leq
            \ChiMinus
            \leq
            2 \Genus - 2\apply{\NumberOfNonsphericalComponents}{\HomologyClass}
        \right\}.
      $
  \end{enumerate}
\end{proposition}
\begin{proof}
  If
  $
    \apply{\ChiMinus}{\Surface}
    <
    2 \apply{\Genus}{\Surface}
    -
    2 \apply{\NumberOfNonsphericalComponents}{\HomologyClass}
  $,
  then there have to exist two non-spherical components of
  $\Surface$
  that get mapped to the same connected component of
  $\TopologicalSpace$.
  Therefore one can take the connected sum of these components.
  This proves Part~(\ref{itm:ConditionalSaturation}).

  For Part~(\ref{itm:Bounds}), it is trivial that
  $
    \apply{\ChiMinus}{\HomologyClass}
    \leq
    \ChiMinus
  $,
  and Lemma~\ref{lem:EssentialSupport} shows that
  $
    \ChiMinus
    \leq
    2 \Genus -2 \apply{\NumberOfNonsphericalComponents}{\HomologyClass}
  $.
  By the saturation move of Part~(\ref{itm:ConditionalSaturation}),
  for any point
  $
    \left( \ChiMinus, \Genus \right)
    \in
    \AttainableSet{\TopologicalSpace}{\HomologyClass}
  $,
  the point
  $
    \left(
      2\Genus - 2\apply{\NumberOfNonsphericalComponents}{\HomologyClass},
      \Genus
    \right)
  $
  is contained in $\AttainableSet{\TopologicalSpace}{\HomologyClass}$ as well.
  This point lies on the ray defining $\apply{\Genus_{c}}{\HomologyClass}$.
  Therefore we have
  $
   \apply{\Genus_{c}}{\HomologyClass}
   \leq
   \apply{\Genus}{\Surface}
  $.
\end{proof}

The saturation moves from Lemma~\ref{lem:AttainableSetSaturation}
and Proposition~\ref{prp:AttainableSetBounds},
Part~(\ref{itm:ConditionalSaturation}),
together with the inequalities of
Proposition~\ref{prp:AttainableSetBounds},
Part~(\ref{itm:Bounds}), are depicted
in Figure~\ref{fig:GenericAttainableSet}.

\begin{figure}[h]%
  \begin{tikzpicture}[scale=0.5]
   \clip (-2,-2) rectangle (20.6,20.6);
   \draw[->] (0,0) --  (20.5,0) node[below]{$\chi^-$};
   \draw[->] (0,0) -- (0,20.5) node[below left]{$g$};
   \filldraw[white!80!black] (2,22) -- (2,18) -- (4,18) --
   (4,14)--(6,14)--(6,8)--(10,8)--(34,20)--cycle;
   \draw[dashed] (0,1) node[left]{$1$} -- (38,20);

   \draw[->] (10,15) -- +(0,1);
   \draw[->] (10,15) -- +(2,0);
   \draw[->] (10,15) -- +(2,1);
   \node at (17.5,18.5) {$A(\alpha)$};
   \def\crosslength{0.05} 
   \foreach \y in {18,19,20,21} {
    \foreach \x in {2,4,...,20}{
    \draw (\x+\crosslength,\y+\crosslength) --
    +(-2*\crosslength,-2*\crosslength);
    \draw (\x-\crosslength,\y+\crosslength) --
    +(2*\crosslength,-2*\crosslength);
    }
   };
   \foreach \y in {14,15,16,17} {
     \pgfmathsetmacro{\fofy}{2*\y-6}
     \foreach \x in {4,6,...,20}{
          \ifthenelse{\x < \fofy \OR \x=\fofy} 
          {
              \draw (\x+\crosslength,\y+\crosslength) --
              +(-2*\crosslength,-2*\crosslength);
              \draw (\x-\crosslength,\y+\crosslength) --
              +(2*\crosslength,-2*\crosslength);
          }{}
     }
   };
   \foreach \y in {8,9,...,13} {
        \pgfmathsetmacro{\fofy}{2*\y-6}
        \foreach \x in {6,8,...,20}{
            \ifthenelse{\x < \fofy \OR \x=\fofy} 
            {
                \draw (\x+\crosslength,\y+\crosslength) --
                +(-2*\crosslength,-2*\crosslength);
                \draw (\x-\crosslength,\y+\crosslength) --
                +(2*\crosslength,-2*\crosslength);
            }{}
        }
    };
   \draw (2,22) -- (2,0) node[below]{$\chi^-(\alpha)$};
   \draw (10,8) -- (0,8) node [left]{$g_c(\alpha)$};
   \node[left] at (0,3) {$n_0(\alpha)$};
   \draw (0,3)--(34,20);
   \filldraw[black] (0,1) circle (0.05);
\end{tikzpicture}
  \caption{%
    An attainable set with saturation moves and lines depicting
    the bounds from Proposition~\ref{prp:AttainableSetBounds}.
  }
  \label{fig:GenericAttainableSet}
\end{figure}
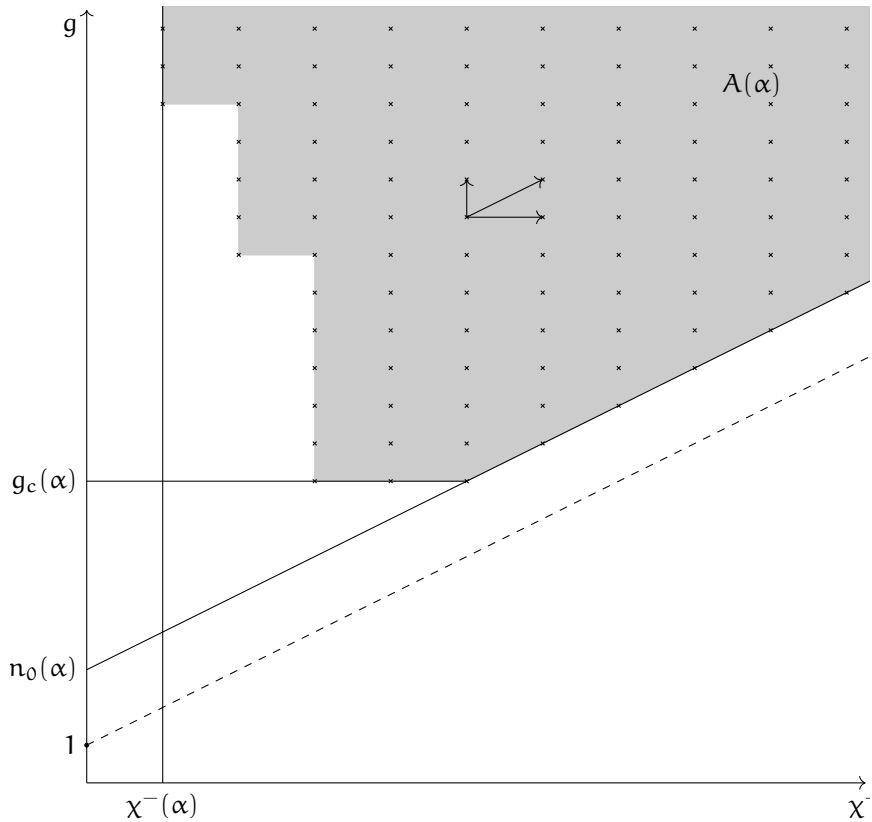

\subsection{Stabilization}
The relationship between $\apply{\Genus}{\HomologyClass}$
and $\apply{\ChiMinus}{\HomologyClass}$ strengthens
for multiples of $\HomologyClass$.
Note that the limit  $\lim_{n\to\infty}
\frac{\apply{\ChiMinus}{n\HomologyClass}}{n}$
exists, since taking covers of a representative shows
that the sequence is monotonically decreasing.
Let us denote the limit by $\apply{\ChiMinus_{\text{stable}}}{\HomologyClass}$.
Similarly, we define $\apply{\Genus_{\text{stable}}}{\HomologyClass}$.
One has the following well known stabilization phenomenon:
\begin{lemma}%
\label{lem:StabilizationChiMinusGenus}
  For every $\HomologyClass\in \HomologyOfSpaceObject{\TopologicalSpace}{2}$,
  one has
  $
    \apply{\ChiMinus_{\text{stable}}}{\HomologyClass}
    =
    2\apply{\Genus_{\text{stable}}}{\HomologyClass}
  $.
\end{lemma}

This stabilization phenomenon can be derived from a stabilization result
for the attainable set.
Let us denote the \introduce{stable attainable set} (set of limit points)
$
  \lim_{n\to\infty}\frac{1}{n}
  \AttainableSet{\TopologicalSpace}{n\HomologyClass}
$
by $\StableAttainableSet{\TopologicalSpace}{\HomologyClass}$.
\begin{lemma}%
\label{lem:AttainableSetStabilization}
  For any $\HomologyClass \in \HomologyOfSpaceObject{\TopologicalSpace}{2}$,
  one has
  \begin{equation}%
  \label{eqn:AttainableSetStable}
    \StableAttainableSet{\TopologicalSpace}{\HomologyClass}=
    \left\{
      \left( \ChiMinus, \Genus \right)
      \in
      \Reals^{2}
    \middle\vert
      \apply{\ChiMinus_{\text{stable}}}{\HomologyClass}
      \leq
      \ChiMinus
      \leq
      2\Genus
    \right\}.
  \end{equation}
\end{lemma}
\begin{proof}
  If a multiple of $\HomologyClass$ is representable by spheres, then
  $\apply{\ChiMinus_{\text{stable}}}{\HomologyClass}$ is zero and the claim
  follows trivially.
  Let us assume for the rest of the proof that no multiple of
  $\HomologyClass$ is representable by spheres.

  For every representative $\Surface$ of $n \HomologyClass$ one has
  $
    n \apply{\ChiMinus_{\text{stable}}}{\HomologyClass}
    \leq
    \apply{\ChiMinus}{\Surface}
    \leq
    2\apply{\Genus}{\Surface}
  $.
  Hence $\StableAttainableSet{\TopologicalSpace}{\HomologyClass}$ is
  contained in the right hand side of Equation~(\ref{eqn:AttainableSetStable}).
  It remains to show the reverse inclusion.

  Given a point
  $
    \left(\ChiMinus, \Genus \right)
    \in
    \StableAttainableSet{\TopologicalSpace}{\HomologyClass}
  $,
  the saturation moves described in Lemma~\ref{lem:AttainableSetSaturation}
  imply that $\StableAttainableSet{\TopologicalSpace}{\HomologyClass}$
  contains all points between the rays
  $\left( \ChiMinus, \Genus + r \right)$ and
  $\left( \ChiMinus + 2r, \Genus + r \right)$ for $r\in \Reals_{\geq 0}$.
  Therefore, it suffices to show that
  $
    \left(
      \apply{\ChiMinus_{\text{stable}}}{\HomologyClass},
      0.5\ \apply{\ChiMinus_{\text{stable}}}{\HomologyClass}
    \right)
    \in
    \StableAttainableSet{\TopologicalSpace}{\HomologyClass}
  $,
  i.e., that there is a sequence $\Surface_{n}$ representing $n\HomologyClass$
  with
  $
    \apply{\ChiMinus_{\text{stable}}}{\HomologyClass}
    =
    \lim_{n\to \infty}
    \frac{\apply{\ChiMinus}{\Surface_{n}}}{n}
    =
    \lim_{n\to \infty}
    2\frac{\apply{\Genus}{\Surface_{n}}}{n}
  $.

  Let $\Surface_{n}$ denote a representative of $n\HomologyClass$
  such that
  $
    \apply{\ChiMinus}{\Surface_{n}}
    =
    \apply{\ChiMinus}{n\HomologyClass}
  $.
  Let $\Surface_{n,m}$ denote an $m$-fold covering space of $\Surface_{n}$
  such that $\Surface_{n,m}$
  has the same number of connected components as $\Surface_{n}$.
  This implies that
  $
    m \apply{\ChiMinus}{\Surface_{n}}
    =
    \apply{\ChiMinus}{\Surface_{n,m}}
    =
    2\apply{\Genus}{\Surface_{n,m}}
    -
    2\apply{\NumberOfNonsphericalComponents}{\Surface_{n}}
  $.
  We conclude that,
  for $\apply{m}{n} = \apply{\NumberOfNonsphericalComponents}{\Surface_{n}}$,
  one has
  \[
    \apply{\ChiMinus_{\text{stable}}}{\HomologyClass}
    =
    \lim_{n\to \infty}
    \frac{\apply{\ChiMinus}{\Surface_{n,\apply{m}{n}}}}{n\cdot\apply{m}{n}}
    =
    \lim_{n\to \infty}
    \frac{%
      2\apply{\Genus}{\Surface_{n,\apply{m}{n}}}
      -
      2\apply{\NumberOfNonsphericalComponents}{\Surface_{n}}
    }{n\cdot\apply{\NumberOfNonsphericalComponents}{\Surface_{n}}}
    =
    \lim_{n\to \infty}
    \frac{%
      2\apply{\Genus}{\Surface_{n,\apply{m}{n}}}
    }{n\cdot\apply{m}{n}}.
  \]
\end{proof}

\begin{proof}[Proof of Lemma~\ref{lem:StabilizationChiMinusGenus}]
  By Lemma~\ref{lem:AttainableSetStabilization} we can find a sequence
  $\Surface_{n}$ representing $n\HomologyClass$ such that
  $
    \lim_{n\to\infty} \frac{2\apply{\Genus}{\Surface_{n}}}{n}
    =
    \apply{\ChiMinus_{\text{stable}}}{\HomologyClass}
  $.
  Therefore we have the following chain of inequalities:
  \[
    \apply{\ChiMinus_{\text{stable}}}{\HomologyClass}
    =
    \lim_{n\to\infty} \frac{2\apply{\Genus}{\Surface_{n}}}{n}
    \geq
    \lim_{n\to\infty} \frac{2\apply{\Genus}{n\HomologyClass}}{n}
    \geq
    \lim_{n\to\infty} \frac{\apply{\ChiMinus}{n\HomologyClass}}{n}
    =
    \apply{\ChiMinus_{\text{stable}}}{\HomologyClass}
  \]
  Since both outer terms are the same, we conclude that all four terms
  are equal.
\end{proof}

Lemma~\ref{lem:AttainableSetStabilization}
implies that the delicate structure
of $\AttainableSet{\TopologicalSpace}{\HomologyClass}$,
i.e.\ the finite difference to the cone described in Equation~(\ref{eqn:Cone}),
vanishes when passing to the stabilization of the attainable set.

\subsection{Simplicial Volume and Further Invariants}
From the attainable set of a homology class $\HomologyClass$ one can recover
$\apply{\Genus}{\HomologyClass}$ as
$
  \min
  \left\{
    \Genus
  \middle\vert
    \left( \ChiMinus, \Genus \right)
    \in
    \AttainableSet{\TopologicalSpace}{\HomologyClass}
  \right\}
$
and $\apply{\ChiMinus}{\HomologyClass}$ as
$
  \min
  \left\{
    \ChiMinus
  \middle\vert
    \left( \ChiMinus, \Genus \right)
    \in
    \AttainableSet{\TopologicalSpace}{\HomologyClass}
  \right\}
$.
One can generalize this by taking the minimum of a linear combination of
$\Genus$ and $\ChiMinus$:
\[
  \apply{l_{p,q}}{\HomologyClass}
  =
  \min
  \left\{
    p \ChiMinus + q \Genus
  \middle\vert
    \left( \ChiMinus, \Genus \right)
    \in
    \AttainableSet{\TopologicalSpace}{\HomologyClass}
  \right\}
\]
For example, the simplicial volume%
\footnote{%
  here: minimal number of triangles in an \emph{integral} simplicial
  representation, ignoring sphere components
}
is $l_{1,2}$.
By the description of the attainable set in Equation~(\ref{eqn:Cone}),
$l_{p,q}$ is finite if and only if $\left(p,q\right)$ lies in the dual cone
\[
  \left\{
    \left(p,q\right)\in \Reals^{2}
    \middle\vert
    0\leq q \and 2p+q\geq 0
  \right\}
\]
These $l_{p,q}$ are convex and positively homogeneous in $\left(p,q\right)$.
As a corollary of Proposition~\ref{prp:AttainableSetBounds} we have:
\begin{lemma}
 If $p\leq 0$ and $2p+q\geq 0$, then
  $
    \apply{l_{p,q}}{\HomologyClass}
    =
    \left(p+2q\right) \apply{\Genus}{\HomologyClass}
    -
    2p \apply{\NumberOfNonsphericalComponents}{\HomologyClass}
  $.
\end{lemma}
\begin{proof}
  For a representative $\Surface$ of $\HomologyClass$ with
  $
    \left( \ChiMinus , \Genus \right)
    =
    \left( \apply{\ChiMinus}{\Surface} , \apply{\Genus}{\Surface} \right)
  $, one has
  \[
    p \ChiMinus + q \Genus
    \geq
    p
    \left(
      2 \Genus
      -
      2 \apply{\NumberOfNonsphericalComponents}{\HomologyClass}
    \right)
    +
    q \Genus
    \geq
    \left( 2 p + q \right) \apply{\Genus}{\HomologyClass}
    -
    2 p \apply{\NumberOfNonsphericalComponents}{\HomologyClass},
  \]
  and for a representative of $\left(\ChiMinus_{c},\Genus_{c}\right)$ we have
  equalities.
\end{proof}

Apart from these restraints, the $l_{p,q}$ seem to be independent.

\section{The Cokernel of the Hurewicz Map}
\label{scn:Hurewicz}
In this section we show that the attainable sets of homology classes of
a topological space $\TopologicalSpace$ are determined by
the fundamental group of the space.
This fact is essential because it allows us to restrict our attention to
$2$-complexes.
More precisely we show that the attainable set of a homology class is
determined by its image in
$\HomologyOfGroupObject{\HomotopyGroupOfObject{\TopologicalSpace}{\Point}{1}}{2}$
under the natural morphism
$
  \HomologyOfSpaceObject{\TopologicalSpace}{2}
  \to
  \HomologyOfGroupObject{\HomotopyGroupOfObject{\TopologicalSpace}{\Point}{1}}{2}
$.
Secondly, we show in Lemma~\ref{lem:MinimizerSurjectiveAttainableSet} and
Corollary~\ref{crl:AttainableSetFundamentalGroup}
how to relate the attainable sets between $\TopologicalSpace$ and
$\EMSpace{{\HomotopyGroupOfObject{\TopologicalSpace}{x}{1}}}{1}$.

For every path-connected pointed CW-complex
$\left(\TopologicalSpace, x\right)$
we have a Hurewicz-homomorphism
\[
  h_{2}
  \colon
  \HomotopyGroupOfObject{\TopologicalSpace}{x}{2}
  \to
  \HomologyOfSpaceObject{\TopologicalSpace}{2}
\]
whose cokernel is independent of the base point $x$.
An important observation
in~\cite{HopfFundamentalgruppe} is the following lemma.
\begin{lemma}%
\label{lem:CokernelGroupHomology}
  For every path-connected pointed CW-complex $\left(X,x\right)$,
  the natural sequence
  \[
    \HomotopyGroupOfObject{\TopologicalSpace}{x}{2}
    \stackrel{h_{2}}{\to}
    \HomologyOfSpaceObject{\TopologicalSpace}{2}
    \to
    \HomologyOfGroupObject
      {\HomotopyGroupOfObject{\TopologicalSpace}{x}{1}}
      {2}
    \to
    0
  \]
  is exact, hence
  $
    \HomologyOfGroupObject
      {\HomotopyGroupOfObject{\TopologicalSpace}{x}{1}}
      {2}
  $
  is the cokernel of the Hurewicz homomorphisms.
\end{lemma}
\begin{proof}
  By gluing in cells of dimension $3$ and higher we
  can include $\TopologicalSpace$ into an Eilenberg-MacLane space
  $
    \iota
    \colon
    \TopologicalSpace
    \to
    \EMSpace{\HomotopyGroupOfObject{\TopologicalSpace}{x}{1}}{1}
  $
  with $\iota$ inducing an isomorphism on fundamental groups.
  Let us abbreviate
  $\EMSpace{\HomotopyGroupOfObject{\TopologicalSpace}{x}{1}}{1}$
  by $\TopologicalSpace_{1}$.

  Now we have the following diagram with exact rows
  and vertical maps given by Hurewicz-homomorphisms:
  \begin{center}
    \begin{tikzcd}
        0
        \ar[r]
      &
        \HomotopyGroupOfPairObject%
          {\TopologicalSpace_{1}}
          {\TopologicalSpace}
          {x}
          {3}
        \ar[r,"\sim"]
        \ar[d,"h_{3}",->>]
      &
        \HomotopyGroupOfObject%
          {\TopologicalSpace}
          {x}
          {2}
        \ar[r]
        \ar[d,"h_{2}"]
      &
        0
        \ar[d]
    \\
      &
        \HomologyOfSpacePairObject%
          {\TopologicalSpace_{1}}
          {\TopologicalSpace}
          {3}
        \ar[r]
      &
        \HomologyOfSpaceObject%
          {\TopologicalSpace}
          {2}
        \ar[r,"\HomologyOfSpaceMorphism{\iota}{2}",->>]
      &
        \HomologyOfSpaceObject%
          {\TopologicalSpace_{1}}
          {2}
    \end{tikzcd}
  \end{center}

  The morphism $\HomologyOfSpaceMorphism{\iota}{2}$ in the diagram is surjective
  because $\iota$ is given by adding cells of dimension $3$ and higher.
  The morphism $h_3$ is surjective by the relative Hurewicz theorem.
  Hence $\HomologyOfSpaceMorphism{\iota}{2}$ is the cokernel of $h_{2}$.
\end{proof}

\begin{lemma}%
\label{lem:AttainableSetDescendsToCokernel}
  For every path-connected pointed CW-complex
  $\left(\TopologicalSpace, x \right)$
  and any two classes
  $
    \HomologyClass_{1}, \HomologyClass_{2}
    \in
    \HomologyOfSpaceObject{\TopologicalSpace}{2}
  $
  which agree in
  $
    \HomologyOfGroupObject
      {\HomotopyGroupOfObject{\TopologicalSpace}{x}{1}}
      {2}
  $
  one has
  $
    \AttainableSet{\TopologicalSpace}{\HomologyClass_{1}}
    =
    \AttainableSet{\TopologicalSpace}{\HomologyClass_{2}}.
  $
\end{lemma}
\begin{proof}
  Because $\HomologyClass_{1}$ and $\HomologyClass_{2}$
  agree in
  $
    \HomologyOfGroupObject
      {\HomotopyGroupOfObject{\TopologicalSpace}{x}{1}}
      {2}
  $,
  by Lemma~\ref{lem:CokernelGroupHomology} there exists a map
  $\ContinuousMap \colon \Sphere \to \TopologicalSpace$
  such that
  $
    \HomologyClass_{1}
    +
    \apply
      {\HomologyOfSpaceMorphism{\ContinuousMap}{2}}
      {\FundamentalClass{\Sphere}}
    =
    \HomologyClass_{2}
  $.
  This implies that
  $
    \AttainableSet{\TopologicalSpace}{\HomologyClass_{1}}
    \subseteq
    \AttainableSet{\TopologicalSpace}{\HomologyClass_{2}}.
  $
  Similarly one obtains the other inclusion.
\end{proof}

Together with Lemma~\ref{lem:AttainableSetDescendsToCokernel},
this shows that the attainable set only depends on the induced class
in the group homology of the fundamental group, and $\TopologicalSpace$.
Therefore we extend the notation of attainable set and write
$\AttainableSet{\TopologicalSpace}{\HomologyClass}$ for a class
$
  \HomologyClass
  \in
  \HomologyOfGroupObject
    {\HomotopyGroupOfObject{\TopologicalSpace}{x}{1}}
    {2}
$
meaning the attainable set of any preimage of $\HomologyClass$ in the homology
of $\TopologicalSpace$.

Now we want to realize this homological statement on the level of spaces.
Therefore we analyze how attainable sets, and more generally, the compression
preorder, behave under maps between topological spaces.
\begin{definition}[Minimizer Surjective]
  We call a map
  $
    \ContinuousMap
    \colon
    \TopologicalSpace_{1}
    \to
    \TopologicalSpace_{2}
  $
  between CW-complexes \introduce{minimizer surjective}, if
  \begin{enumerate}[(a)]
    \item
      $
        \HomologyOfSpaceMorphism{\ContinuousMap}{2}
        \colon
        \HomologyOfSpaceObject{\TopologicalSpace_{1}}{2}
        \to
        \HomologyOfSpaceObject{\TopologicalSpace_{2}}{2}
      $
      is surjective and
    \item
      for every
      $\HomologyClass \in \HomologyOfSpaceObject{\TopologicalSpace_{1}}{2}$
      and every closed, oriented surface
      $\ContinuousMap_{2} \colon \Surface_{2} \to \TopologicalSpace_{2}$
      representing
      $
        \apply
          {\HomologyOfSpaceMorphism{\ContinuousMap}{2}}
          {\HomologyClass}
      $
      there exists a closed, oriented surface
      $\ContinuousMap_{1} \colon \Surface_{1} \to \TopologicalSpace_{1}$
      representing $\HomologyClass$ and
      $
        \left(\Surface_{1}, \ContinuousMap \circ \ContinuousMap_{1}\right)
        \CompressionSmallerEqual
        \left(\Surface_{2}, \ContinuousMap_{2}\right)
      $.
  \end{enumerate}
\end{definition}

Being minimizer surjective is invariant under homotopy.
Note that being minimizer surjective implies that one can lift minimizers of
$\CompressionSmallerEqual$ of representatives of a homology class up to
equivalence to representatives of every class in its preimage.
The following lemma shows that a minimizer surjective map
allows to shift computing attainable sets between its domain and codomain in
both directions.
\begin{lemma}%
\label{lem:MinimizerSurjectiveAttainableSet}
  Given a minimizer surjective map
  $
    \ContinuousMap
    \colon
    \TopologicalSpace_{1}
    \to
    \TopologicalSpace_{2}
  $
  between path-connected CW-complexes, for every
  $\HomologyClass \in \HomologyOfSpaceObject{\TopologicalSpace_{1}}{2}$
  one has
  $
    \AttainableSet{\TopologicalSpace_{1}}{\HomologyClass}
    =
    \AttainableSet%
      {\TopologicalSpace_{2}}
      {\apply{\HomologyOfSpaceMorphism{\ContinuousMap}{2}}{\HomologyClass}}
  $.
\end{lemma}
\begin{proof}
  Since postcomposing a representative of $\HomologyClass$ by $\ContinuousMap$
  yields a representative of
  $\apply{\HomologyOfSpaceMorphism{\ContinuousMap}{2}}{\HomologyClass}$, one
  has
  $
  \AttainableSet{\TopologicalSpace_{1}}{\HomologyClass}
  \subset
  \AttainableSet{\TopologicalSpace_{2}}%
    {\apply{\HomologyOfSpaceMorphism{\ContinuousMap}{2}}{\HomologyClass}}
  $.

  Now fix a point
  $
    \left(\ChiMinus_{2},\Genus_{2}\right)
    \in
    \AttainableSet{\TopologicalSpace_{2}}%
      {\apply{\HomologyOfSpaceMorphism{\ContinuousMap}{2}}{\HomologyClass}}
  $%
  .
  If $\ChiMinus_{2}=2\Genus_{2}=0$, then
  $
    \apply{\HomologyOfSpaceMorphism{\ContinuousMap}{2}}{\HomologyClass}
  $
  is represented by a sphere and, because $\ContinuousMap$ is minimizer
  surjective, $\HomologyClass$ is also represented by a sphere,
  hence
  $
    \left(0,0\right)
    \in
    \AttainableSet{\TopologicalSpace_{1}}{\HomologyClass}
  $%
  .
  Otherwise,
  $
    \apply{\NumberOfNonsphericalComponents}%
      {\apply{\HomologyOfSpaceMorphism{\ContinuousMap}{2}}{\HomologyClass}}
    =1
  $%
  , which implies
  $\ChiMinus_{2} \leq 2 \Genus_{2} - 2$.
  Since $\ContinuousMap$ is minimizer surjective, there exists a point
  $
    \left(\ChiMinus_{1},\Genus_{1}\right)
    \in
    \AttainableSet{\TopologicalSpace_{1}}{\HomologyClass}
  $
  such that
  $\ChiMinus_{2}\geq \ChiMinus_{1}$
  and
  $\Genus_{2}\geq \Genus_{1}$.
  Lemma~\ref{lem:AttainableSetSaturation} implies that
  $
    \left(\ChiMinus_{1},\Genus_{2}\right)
    \in
    \AttainableSet{\TopologicalSpace_{1}}{\HomologyClass}
  $%
  , and Proposition~\ref{prp:AttainableSetBounds}~(\ref{itm:ConditionalSaturation})
  together with
  $\ChiMinus_{2} \leq 2 \Genus_{2} - 2$
  and
  $
    \apply{\NumberOfNonsphericalComponents}%
    {\HomologyClass}
    \leq
    1
  $
  implies that
  $
    \left(\ChiMinus_{2},\Genus_{2}\right)
    \in
    \AttainableSet{\TopologicalSpace_{1}}{\HomologyClass}
  $.
\end{proof}

\begin{proposition}%
[Canonical Map to $\EMSpace{\HomotopyGroupOfObject{X}{x}{1}}{1}$
is minimizer surjective]%
\label{prp:MinimizerSurjectiveFundamentalGroup}
  For a path-connected pointed CW-complex
  $\left(\TopologicalSpace,x\right)$
  the canonical inclusion
  $
    \iota
    \colon
    \TopologicalSpace
    \to
    \EMSpace{\HomotopyGroupOfObject{X}{x}{1}}{1}
  $
  is minimizer surjective.
\end{proposition}
\begin{proof}
  The space $\TopologicalSpace$ includes into
  $\EMSpace{\HomotopyGroupOfObject{X}{x}{1}}{1}$
  such that the $2$-skeleton of $\TopologicalSpace$
  agrees with the $2$-skeleton of
  $\EMSpace{\HomotopyGroupOfObject{X}{x}{1}}{1}$.
  The cellular approximation theorem implies
  that we can lift every map from a surface to
  $\EMSpace{\HomotopyGroupOfObject{X}{x}{1}}{1}$
  along $\iota$ up to homotopy.
\end{proof}

\begin{corollary}[Attainable set determined by fundamental group]%
\label{crl:AttainableSetFundamentalGroup}
  Given two path-connected pointed CW-complexes
  $\left(\TopologicalSpace_{1}, x_{1}\right)$,
  $\left(\TopologicalSpace_{2}, x_{2}\right)$,
  and an isomorphism
  $
    f
    \colon
    \HomotopyGroupOfObject{\TopologicalSpace_{1}}{x_{1}}{1}
    \stackrel{\sim}{\to}
    \HomotopyGroupOfObject{\TopologicalSpace_{2}}{x_{2}}{1}
  $,
  the induced isomorphism
  \[
    \HomologyOfGroupObject%
      {\HomotopyGroupOfObject{\TopologicalSpace_{1}}{x_{1}}{1}}
      {2}
    \stackrel%
      {\HomologyOfGroupMorphism{f}{2}}
      {\to}
    \HomologyOfGroupObject%
      {\HomotopyGroupOfObject{\TopologicalSpace_{2}}{x_{2}}{2}}
      {2}
  \]
  preserves attainable sets, i.e., for every
  $
    \HomologyClass
    \in
    \HomologyOfGroupObject%
      {\HomotopyGroupOfObject{\TopologicalSpace_{1}}{x_{1}}{1}}
      {2}
  $,
  one has
  $
    \AttainableSet{\TopologicalSpace_{1}}{\HomologyClass}
    =
    \AttainableSet{\TopologicalSpace_{2}}%
      {\apply{\HomologyOfGroupMorphism{f}{2}}{\HomologyClass}}
  $.
\end{corollary}

Because every group can be realized as the fundamental group of a $2$-complex,
this implies that the task of computing attainable sets can be reduced to the
case of $2$-complexes.
\begin{example}[Hypercube]%
\label{exm:nTorus}
  Consider $\apply{K}{\Integers^{n},1} = {\left(\Circle\right)}^{n}$
  and let  $\TopologicalSpace$ be its $2$-skeleton
  (in the product cell structure where $\Circle$
  has exactly one $0$-cell and one $1$-cell).
  Then the inclusion of $\TopologicalSpace$ into
  $\apply{K}{\Integers^{n},1}$ is minimizer surjective.
  In this case, we can say a bit more about the minimal genus.
  There are canonical isomorphisms
  $
    \HomologyOfSpaceObject{\TopologicalSpace}{1}
    \cong
    \Integers^{n}
  $
  and
  $
    \HomologyOfSpaceObject{\TopologicalSpace}{2}
    \cong
    \bigwedge^{2} \Integers^{n}
  $.
  Because $\HomotopyGroupOfObject{\TopologicalSpace}{x}{1}$ is abelian, any
  surface $\Surface$ representing a homology class $\HomologyClass$ can
  be compressed to a disjoint union of tori with the same genus.
  Hence $\apply{\Genus}{\HomologyClass}$ is the minimal number of summands
  in a decomposition of $\HomologyClass$ into elementary wedges:
  \[
    \HomologyClass
    =
    \sum_{i=1}^{\apply{\Genus}{\HomologyClass}}
      v_{i}\wedge w_{i}
    , \quad
    v_{i}, w_{i} \in \Integers^{n}
  \]
\end{example}

The following proposition is an easy corollary from
Corollary~\ref{crl:CellwiseCovering}
\begin{proposition}[Free Product is Minimizer Surjective]
  The quotient map
  $
    \EMSpace{\Group_{1}}{1} \sqcup \EMSpace{\Group_{2}}{1}
    \to
    \EMSpace{\Group_{1} \ast \Group_{2}}{1}
  $
  is minimizer-surjective.
\end{proposition}

\section{Generalized 2-Complexes}
\label{scn:TwoComplexes}
In this section we generalize the definition of 2-complexes
by allowing the cells to have nonzero genus.
Technically, this is not necessary for statements about computability,
as all subsequent statement hold for usual 2-complexes,
and any generalized 2-complex can be subdivided to be a usual 2-complex.
However, this subdivision increases the number of cells significantly
and complicates the subsequent computations for given spaces one might be
interested in.
Hence we introduce this notion to allow for concrete computations that
could not be handled without the use of a computer otherwise.

\begin{definition}[Manifold Models]%
\label{dfn:ManifoldModels}
  For each dimension $0,1,2$, we fix models for each diffeomorphism
  type of compact, connected manifolds, which are either a point or have
  non-empty boundary:
  \[
    \ManifoldModelsOfDimension{0} = \{ \Point \}
    \quad
    \text{a point}
  \]
  \[
    \ManifoldModelsOfDimension{1} = \{ \Interval \}
    \quad
    \text{an interval}
  \]
  \[
    \ManifoldModelsOfDimension{2} =
      \{
        \SurfaceOfGBO{\Genus}
                      {\NumberBoundaryComponents}
                      {\OrientationIndicator}
          \mid b \geq 1
      \}
      \xrightarrow[(\Genus,\NumberBoundaryComponents,\OrientationIndicator)]%
                  {\sim}
      \mathbb{N}_{0} \times \mathbb{N}_{\geq 1} \times
      \{ \Orientable, \NotOrientable \}.
  \]
  where
  $\SurfaceOfGBO{\Genus}{\NumberBoundaryComponents}{\OrientationIndicator}$
  denotes a compact surface of genus $\Genus$ with
  $\NumberBoundaryComponents$ boundary components and $\OrientationIndicator$
  indicating whether it is orientable.

  Additionally, we fix
  \begin{itemize}
    \item
      for each orientable model
      $
        \Manifold \in
        \left\{
          \Point, \Interval,
          {\left(
            \SurfaceOfGBO{\Genus}{\NumberBoundaryComponents}{\Orientable}
          \right)}_{\Genus,\NumberBoundaryComponents}
        \right\}
      $
      an orientation,
    \item
      for each model $\Manifold$ a closed collar
      $%
        \CollarMap{\Manifold} \colon
        \partial \Manifold \times \left[ 0,1 \right] \to
        \Manifold
      $.
  \end{itemize}
\end{definition}

\begin{definition}[Generalized 2-Complex]
\label{dfn:Generalized2Complex}
  A \introduce{generalized 2-complex} $\GeneralizedComplex$
  is given by the following data:
  \begin{enumerate}[(i)]
    \item
      Three sets
      $\CellsOfDim{\GeneralizedComplex}{0}$,
      $\CellsOfDim{\GeneralizedComplex}{1}$,
      $\CellsOfDim{\GeneralizedComplex}{2}$,
      whose elements are called \introduce{cells} of dimension
      $0$, $1$, $2$, respectively,
    \item
      maps $M$ indicating the \introduce{cell type}
      \[
        \CellType \colon
          \CellsOfDim{\GeneralizedComplex}{0} \to
          \ManifoldModelsOfDimension{0}, \\
        \CellType \colon
          \CellsOfDim{\GeneralizedComplex}{1} \to
          \ManifoldModelsOfDimension{1}, \\
        \CellType \colon
          \CellsOfDim{\GeneralizedComplex}{2} \to
          \ManifoldModelsOfDimension{2},
      \]
    \item
      for each cell
      $\Cell \in \CellsOfDim{\GeneralizedComplex}{\Dimension}$,
      a (continuous) \introduce{gluing map} to the previous skeleton
      (see~(\ref{itm:Generalized2ComplexSkeleton}))
      \[
        \GluingMap{\Cell} \colon
        \partial \CellType \left( \Cell \right) \to
        \SkeletonOfDim{\GeneralizedComplex}{\Dimension - 1},
      \]
    \item\label{itm:Generalized2ComplexSkeleton}
      for each dimension $\Dimension \in \{-1,0,1,2\}$,
      a \introduce{skeleton}
      $\SkeletonOfDim{\GeneralizedComplex}{\Dimension}$
      which is, starting with
      $\SkeletonOfDim{\GeneralizedComplex}{-1} = \emptyset$,
      an inductive choice of a pushout (in the category of topological spaces)
      for the following diagram:
      \begin{displaymath}
      \begin{tikzcd}
        \coprod_{\Cell \in \CellsOfDim{\GeneralizedComplex}{\Dimension}}
          \partial \CellType \left( \Cell \right)
        \arrow{d}
        \arrow[rr,"{
          \coprod_{\Cell \in \CellsOfDim{\GeneralizedComplex}{\Dimension}}
          \GluingMap{\Cell}}"]
        &&
        \SkeletonOfDim{\GeneralizedComplex}{\Dimension - 1}
        \arrow[d,dotted]
        \\
        \coprod_{\Cell \in \CellsOfDim{\GeneralizedComplex}{\Dimension}}
          \CellType \left( \Cell \right)
        \arrow[rr,dotted,"
          \coprod_{\Cell \in \CellsOfDim{\GeneralizedComplex}{\Dimension}}%
          \CellInclusion{\Cell}%
        "]
        &&
        \SkeletonOfDim{\GeneralizedComplex}{\Dimension}
      \end{tikzcd}
      \end{displaymath}
  \end{enumerate}
  We call $\GeometricRealization{\GeneralizedComplex} =
  \SkeletonOfDim{\GeneralizedComplex}{2}$ the \introduce{geometric realization}.
  The components $\CellInclusion{\Cell}$ of the horizontal dotted arrow are
  called \introduce{cell inclusions}. They are injective on the interior and
  hence define a smooth structure on the image of the interior.
  One can always choose the pushout in such a way that the vertical dotted
  arrow is an inclusion on the underlying sets, we will silently assume
  this whenever it is convenient.
\end{definition}

\begin{lemma}[Homology of a Generalized 2-Complex]
\label{lem:HomologyGeneralized2Complex}
The second homology group of a generalized 2-complex
$\GeneralizedComplex$ can be computed as
\[
  \HomologyOfSpaceObject{\GeometricRealization{\GeneralizedComplex}}{2}
  =
  \kernel{%
    \bigoplus_{\substack{%
        \Cell \in \CellsOfDim{\GeneralizedComplex}{2},\\%
        \Cell \mathrm{\ orientable}}%
      }%
      \HomologyOfSpacePairObject%
        {\CellType\left(\Cell\right)}%
        {\partial\CellType\left(\Cell\right)}%
        {2}
    \rightarrow
    \HomologyOfSpaceObject{\SkeletonOfDim{\GeneralizedComplex}{1}}{1}%
  }.
\]
In particular, if $\iota_O\colon \OrientableComplex \subseteq \GeneralizedComplex$ denotes
the subcomplex of orientable cells, the inclusion
$
  \GeometricRealization{\OrientableComplex}
  \rightarrow
  \GeometricRealization{\GeneralizedComplex}
$
induces an isomorphism
$
  \HomologyOfSpaceObject{\GeometricRealization{\OrientableComplex}}{2}
  \stackrel{\sim}{\rightarrow}
  \HomologyOfSpaceObject{\GeometricRealization{\GeneralizedComplex}}{2}
$.
\end{lemma}
\begin{proof}
This follows from the long exact sequence for the pair
$%
  (%
    \GeometricRealization{\GeneralizedComplex},%
    \SkeletonOfDim{\GeneralizedComplex}{1}%
  )%
$,
the isomorphism
$
  \HomologyOfSpacePairObject%
    {\GeometricRealization{\GeneralizedComplex}}%
    {\SkeletonOfDim{\GeneralizedComplex}{1}}%
    {2}
  \stackrel{\sim}{\rightarrow}
  \bigoplus_{\Cell \in \CellsOfDim{\GeneralizedComplex}{2},\\%
  }%
  \HomologyOfSpacePairObject%
    {\CellType\left(\Cell\right)}%
    {\partial\CellType\left(\Cell\right)}%
    {2}
$,
and the fact that
$
  \HomologyOfSpacePairObject%
    {\CellType\left(\Cell\right)}%
    {\partial\CellType\left(\Cell\right)}%
    {2}
$
vanishes for non-orientable cells $\Cell$.
\end{proof}

\section{Reduction to Cellwise Coverings}
\label{scn:CellwiseCovering}
In this section we prove Theorem~\ref{thm:NormalForm} which extends a theorem
of Edmonds (Theorem 1.1 of \cite{EdmondsBranchedCoverings}, see also
\cite{SkoraDegree} for a version we will use later)
from maps between surfaces to maps from surfaces to generalized 2-complexes.
The theorem of Edmonds states that every map of non-zero degree between closed,
connected surfaces factors up to homotopy as a composition of a pinch map and a
branched covering.
This theorem completely solves the description of the attainable set of
surfaces:
Minimizers do not allow pinches and computing $\Genus$ and $\ChiMinus$ for
branched coverings shows that minimizers have to be non-branched coverings.
An equivalent formulation, which emphasizes its application to the minimizer
problem, is that every map of non-zero degree between closed, connected
surfaces which admits no pinches is homotopic to a branched covering.
Speaking precisely, it is this formulation that we extend in
Theorem~\ref{thm:NormalForm}.

Let us define what a pinch is.
We will also need the related definition of a squeeze (both can be found in
\cite{SkoraDegree}):
\begin{definition}[Squeeze and Pinch]
Let $\TopologicalSpace$ denote a topological space.  A continuous map
  $
    \ContinuousMap
    \colon
    \Surface
    \to
    \TopologicalSpace
  $
  \introduce{admits a squeeze}
  if there exists a non-nullhomotopic simple closed curve in
  $\Surface$ such that $\ContinuousMap$ maps this curve to a nullhomotopic
  curve in $\TopologicalSpace$
  or if $\Surface$ is $\Sphere$ and $\ContinuousMap$ is nullhomotopic.
  It is said to \introduce{admit a pinch} if there exists a closed subsurface
  $\Surface'\subset \Surface$
  with a single boundary component that is not a $2$-disk and such that
  $\at{\ContinuousMap}{\Surface'}$ is nullhomotopic.
\end{definition}

If a map admits no squeezes, it also admits no pinches.
In Theorem~\ref{thm:NormalForm} we require the map to not admit squeezes,
which is a stronger requirement than the one in the cited theorem.
This stems from the fact that maps to generalized
$2$-complexes may possess complicated squeezes.
The reason for considering the equivalence class the compression order instead
of homotopy classes as in the main theorem of \cite{EdmondsBranchedCoverings}
is more subtle and stems from the presence of complicated second
homotopy groups for generalized $2$-complexes on the one
side and restrictions for branched coverings between surfaces
of genus $0$ on the other side.
The result suffices for studying compression order
minimizers, because they cannot admit squeezes and are only
defined in terms of equivalence classes.

In our case, the role of branched coverings in the theorem of Edmonds is taken
by cellwise coverings without folds.
If the target generalized $2$-complex is a surface, these notions are equivalent
for non-zero degree maps without squeezes, and the branch locus is within
the $0$-cells.
\begin{definition}[Cellwise Covering]
\label{dfn:CellwiseCovering}
  We call a map $f\colon \GeometricRealization{X}\to \GeometricRealization{Y}$
  between generalized $2$-complexes a \introduce{cellwise covering} if the
  preimage of every open cell consists of open cells of the same dimension and
  the restriction to this preimage is a covering.
  We call a map $f\colon \Surface \to \GeometricRealization{X}$ a
  \introduce{cellwise covering} if there exists a generalized $2$-complex
  $\mathcal{S}$ and a homeomorphism $\phi \colon
  \GeometricRealization{\mathcal{S}}\to \Surface$ such that $f\circ\phi$ is a
  cellwise covering.

  Furthermore we call a map $f\colon \Surface\to \GeometricRealization{X}$ a
  \introduce{cellwise covering without folds} if it is a cellwise covering and
  additionally the map $f\circ \phi$ is injective in a neighborhood of the inner
  of the $1$-cells of $\mathcal{S}$.
\end{definition}

For later use we need a version of the main result of
\cite{EdmondsBranchedCoverings} for surfaces with boundary
(again see \cite{SkoraDegree} for a strengthened version).
In order to state this we need some more definitions:
\begin{definition}[Geometric Degree, Allowable]
  A map between compact surfaces
  $
    \ContinuousMap
    \colon
    \Surface
    \to
    \Surface'
  $
  is called \introduce{proper} if
  $
    \apply{\ContinuousMap^{-1}}{\partial \Surface'}
    =
    \partial \Surface
  $.
  A homotopy $H$ between proper maps is called \introduce{proper}
  if
  $
    \apply{H^{-1}}{\partial \Surface'}
    =
    \partial \Surface \times I
  $.
  The \introduce{geometric degree} of a proper map
  $
    \ContinuousMap
    \colon
    \Surface
    \to
    \Surface'
  $
  denoted by $\apply{\GeometricDegree}{\ContinuousMap}$ is
  defined as the smallest
  natural number $d$ such that for some $2$-disk $\Ball{2} \subset \Surface'$
  there exists a map $\ContinuousMap'$ that is properly homotopic to
  $\ContinuousMap$ such
  that
  $
    \at{
      \ContinuousMap'
    }{
      \apply{\ContinuousMap'^{-1}}{D}
    }
  $
  is a $d$-fold covering.

  We call a proper map $\ContinuousMap\colon \Surface \to \Surface'$
  \introduce{allowable}
  if
  $
    \at{\ContinuousMap}{\partial \Surface}
  $
  is a $\apply{\GeometricDegree}{\ContinuousMap}$-fold
  covering.
\end{definition}

The relative version of Edmond's result in
\cite{SkoraDegree} states that every allowable map of
non-zero geometric degree between closed surfaces that admits
no pinches is homotopic relative to the boundary to a
branched covering.

\subsection{Block Decomposition}
The most complicated behavior of generalized $2$-complexes and maps
to them occurs in a neighborhood of their $1$-skeleton.
In this subsection, we give models for neighborhoods around points in
the $1$-skeleton (which we call blocks),
and in the following section we give a normal form result for maps
into these neighborhoods.
The block structure depends on the following definition:
\begin{definition}[Plasma, Membrane, Singular Set]%
\label{dfn:PlasmaMembraneSingularSet}
  For a $2$-dimensional model manifold $\CellType$
  as in Definition~\ref{dfn:ManifoldModels},
  we define the \introduce{plasma} $\CellType_{p}$ as the closure of the
  complement of the image of the closed collar,
  $
    \CellType_{p}
    =
    \overline{\CellType \setminus \operatorname{Im} \CollarMap{\CellType}}
   $,
  and the \introduce{membrane} $\CellType_{m}$
  as the boundary of the plasma,
  $
    \CellType_{m}
    =
    \partial \CellType_{p}
  $.

  For a cell $\Cell$ of a generalized $2$-complex $\GeneralizedComplex$,
  we denote by $\IncludedCell{\Cell}$ the image of the cell inclusion,
  by $\IncludedPlasma{\Cell}$ the image of the plasma and
  by $\IncludedMembrane{\Cell}$ the image of the membrane.

  We call the closure of the complement of
  $
    \bigcup_{
      \Cell
      \in
      \CellsOfDim{\GeneralizedComplex}{i+1}
    }
    \IncludedPlasma{\Cell}
  $
  in
  $
    \SkeletonOfDim{\GeneralizedComplex}{i+1}
  $
  the \introduce{$i$-singular set of $\GeneralizedComplex$}
  and denote it by $\apply{\SingularSet{i}}{\GeneralizedComplex}$.
  It is a closed neighborhood of $\SkeletonOfDim{\GeneralizedComplex}{i}$
  inside $\SkeletonOfDim{\GeneralizedComplex}{i+1}$.
\end{definition}

\begin{definition}[Block, Singular Set Projection]%
\label{dfn:BlockSingularSetProjection}
Let $\Cell$ denote a $2$-cell,
then we define the \introduce{$2$-block} $\Block{\Cell}$
to be $\IncludedPlasma{\Cell}$.
The closure of the complement of all $2$-blocks is the $1$-singular set
$
  \apply{\SingularSet{1}}{\GeneralizedComplex}
$.

For $i \in \left\{0,1\right\}$, there is a projection
$
  \SingularSetProjection{i}
  \colon
  \apply{\SingularSet{i}}{\GeneralizedComplex}
  \to
  \SkeletonOfDim{\GeneralizedComplex}{i}
$,
given by the union of the collar projections.
The left square in the following diagram is a pushout
($\pi_{\Cell}$ is the projection onto the first element):
\[
  \begin{tikzcd}
    \bigsqcup_{
      \Cell\in \CellsOfDim{\GeneralizedComplex}{i+1}
    }
    \partial \apply{\ManifoldModels}{\Cell}
    \ar[
        d,
      "\bigsqcup \GluingMap{\Cell}"
    ]
    \ar[
      r,
      "\times \{0\}"
    ]
    &
    \bigsqcup_{
      \Cell\in \CellsOfDim{\GeneralizedComplex}{i+1}
    }
    \partial
    \apply{\ManifoldModels}{\Cell}
    \times
    \Interval
    \ar[
        d,
      "\bigsqcup \CellInclusion{\Cell} \circ \CollarMap{\apply{\CellType}{\Cell}}"
    ]
    \ar[
      r,
      "{\left(\pi_{\Cell}\right)}_{\Cell}"
    ]
    &
    \bigsqcup_{
      \Cell\in \CellsOfDim{\GeneralizedComplex}{i+1}
    }
    \partial \apply{\ManifoldModels}{\Cell}
    \ar[
      d,
      "\bigsqcup \GluingMap{\Cell}"
    ]
    \\
    \SkeletonOfDim{\GeneralizedComplex}{i}
    \ar[r]
    &
    \apply{\SingularSet{i}}{\GeneralizedComplex}
    \ar[
      r,
      "\SingularSetProjection{i}",
      dotted
    ]
    &
    \SkeletonOfDim{\GeneralizedComplex}{i}
  \end{tikzcd}
\]
Therefore we can define the \introduce{singular set projection}
$\SingularSetProjection{i}$ as the pushout of
$
  \bigsqcup
  \GluingMap{\Cell}
  \circ
  \pi_{\Cell}
$
and the identity
$
  \SkeletonOfDim{\GeneralizedComplex}{i}
  \to
  \SkeletonOfDim{\GeneralizedComplex}{i}
$.
For every $\Cell \in \CellsOfDim{\GeneralizedComplex}{1}$
we define the \introduce{$1$-block} $\Block{\Cell}$ to be
$
  \apply%
  {\SingularSetProjection{1}^{-1}}
  {\IncludedPlasma{\Cell}}
$,
and for every
$\Cell \in \CellsOfDim{\GeneralizedComplex}{0}$
we define the \introduce{$0$-block} $\Block{\Cell}$ to be
$
  \apply{
    \left(
      \SingularSetProjection{0}
      \circ
      \SingularSetProjection{1}
    \right)
    ^{-1}
  }
  {\IncludedPlasma{\Cell}}
$.
Note that all blocks are closed,
$\GeneralizedComplex$ is the union of its blocks,
and blocks only intersect in their boundaries.
Two different $i$-blocks, for a fixed $i$, never intersect.
They define a tri-colored decomposition on $\GeneralizedComplex$.
\end{definition}

If the gluing maps of a generalized $2$-complex are too wild, they may force
maps $\Surface\to \GeometricRealization{\GeneralizedComplex}$ to always have
folds.
Hence we need the following definition:
\begin{definition}[Combinatorial Generalized 2-Complex]%
\label{dfn:CombinatorialGeneralized2Complex}
  We call a generalized 2-complex $\GeneralizedComplex$
  \introduce{combinatorial}, if for every $2$-cell $\Cell$ and every connected
  component $C$ of
  $\partial \apply{\CellType}{\Cell}$, one of the following holds:
  \begin{itemize}
      \item
        The restriction $\at{\GluingMap{\Cell}}{C}$ is constant and the image
        is a $0$-cell.
      \item
        The preimage of $\SkeletonOfDim{\GeneralizedComplex}{0}$ under
        $\at{\GluingMap{\Cell}}{C}$ is a finite union of points and
        $\at{\GluingMap{\Cell}}{C}$ maps every component of their complement
        diffeomorphically to an open $1$-cell.
  \end{itemize}
Here the smooth structure on the open $1$-cells stems from the smooth
structure of the underlying cells as explained in
Definition~\ref{dfn:Generalized2Complex}.
\end{definition}
\begin{remark}
  Since homotopic gluing maps give homotopy equivalent geometric realizations,
  every generalized $2$-complex is homotopy equivalent to a combinatorial one.
\end{remark}

\begin{definition}[Block Boundary Graph]%
\label{dfn:BlockBoundaryGraph}
  Let $\GeneralizedComplex$ be a combinatorial generalized $2$-complex.
  The union of all boundaries of blocks can be given a
  graph structure in the following way:
  \begin{itemize}
    \item
      The intersection of a $2$-block and a $1$-block
      is a disjoint union of closed intervals,
    \item
      the intersection of a $2$-block and a $0$-block
      is a disjoint union of closed intervals and circles, and
    \item
      the intersection of a $1$-block and a $0$-block is a disjoint union
      of stars,
      each of which is a graph with the middle point
      and the endpoints of the rays as vertices.
  \end{itemize}
  In total, this decomposes the union of all boundaries of blocks into
  closed intervals and circles which meet only in endpoints of intervals.
  We denote the set of intervals and circles by
  $\BoundaryGraphEdges{\GeneralizedComplex}$
  (each element is a subset of $\GeneralizedComplex$),
  and the set of endpoints of intervals by
  $\BoundaryGraphVertices{\GeneralizedComplex}$.
  Note that the interior of each element of
  $\BoundaryGraphEdges{\GeneralizedComplex}$ is smoothly embedded in a
  $2$-cell.
\end{definition}

Let us assume from here on forth that $\GeneralizedComplex$ is combinatorial.
Given a $1$-cell $\Cell$, the $1$-block belonging to $\Cell$ has
a standard form (See Figure~\ref{fig:Blocks}):
Let $n_{\Cell}$ denote the number of incoming $2$-cells counted with
multiplicity, i.e., the number of connected components of the preimage of
$\IncludedPlasma{\Cell}$ under $\bigsqcup \GluingMap{\Cell'}$, where $\Cell'$
ranges over all $2$-cells. Then $\Block{\Cell}$ is isomorphic to
\[
  \raisebox{.2em}{$
    \Interval
    \times
    \Interval
    \times
    \{1,\ldots,n\}
  $}
  /
    \raisebox{-.2em}{$
    \left(
      0,s,i
    \right)
    \sim
    \left(
      0,s,j
    \right)
  $}
\]
Here the first coordinate parameterizes the collar of the adjacent $2$-cells and
the second coordinate parameterizes $\IncludedPlasma{\Cell}$.

Similarly, given a $0$-cell $\Cell$, we can give a standard form for
its $0$-blocks (See Figure~\ref{fig:Blocks}):
Let $n_{\Cell}$ denote the number of essentially incoming boundaries of
$2$-cells i.e. the number of preimages of
$\IncludedPlasma{\Cell}=\IncludedCell{\Cell}$ under the restriction of all
gluing maps of $2$-cells to their essential components.
Let us silently identify $\{1,\ldots n_{\Cell}\}$ with this preimage.
Let $n^{0}_{\Cell}$ denote the number of boundaries of $2$-cells that get glued
inessentially to $\Cell$.
Again let us silently identify $\{1,\ldots,n^{0}_{\Cell}\}$ with the set of
those boundaries.
Analogously define $k_{\Cell}$ as the number of incoming $1$-cells counted with
multiplicity and let us silently identify $\{1,\ldots k_{\Cell}\}$ with the
incoming edges.
Note that for every boundary point in an essential component of a $2$-cell that
maps to $\Cell$ under the gluing map, a small neighborhood of this boundary
point maps to the incoming $1$-edges. Hence every such point specifies two
(possibly agreeing) elements of $\{1,\ldots,k_{\Cell}\}$.
By fixing an orientation of the boundary of all $2$-cells, we can order these
two elements and hence we get two maps (picking the first resp. the second
element)
\[
  s,t
  \colon
  \left\{
    1,\ldots, n_{\Cell}
  \right\}
  \to
  \left\{
    1,\ldots, k_{\Cell}
  \right\}
\]
Then the $0$-block belonging to $\Cell$ is homeomorphic to
\[
  \left.
    \raisebox{.4em}{$
      \begin{gathered}
        \Interval
        \times
        \Interval
        \times
        \left\{
          1,\ldots,n_{\Cell}
        \right\}
        \cup
        \Interval
        \times
        \left\{
          1,\ldots,k_{\Cell}
        \right\}
        \\
        \cup
        \
        \Interval
        \times
        \Circle
        \times
        \left\{
          1,\ldots,n^0_{\Cell}
        \right\}
      \end{gathered}
    $}
    \middle/
    \raisebox{-25pt}
    {$
      \begin{gathered}
        \left(0,x,i\right)
        =
        \left(0,y,j\right)
        =
        \left(0,l\right)
        =\left(0,z,i\right)
        \\
        \left(x,0,i\right)
        =
        \left(x,\apply{s}{i}\right)
        \\
        \left(x,1,i\right)
        =
        \left(x,\apply{t}{i}\right)
      \end{gathered}
    $}
  \right.
\]
Here $x,y,l\in \Interval$ and $z\in\Circle$ and the first coordinate
parameterizes the collar of the corresponding cells.
We will call these coordinate descriptions of the block the
\introduce{standard form of $\Block{\Cell}$}.
\begin{figure}[h]%
  \def\svgwidth{0.7\textwidth}
  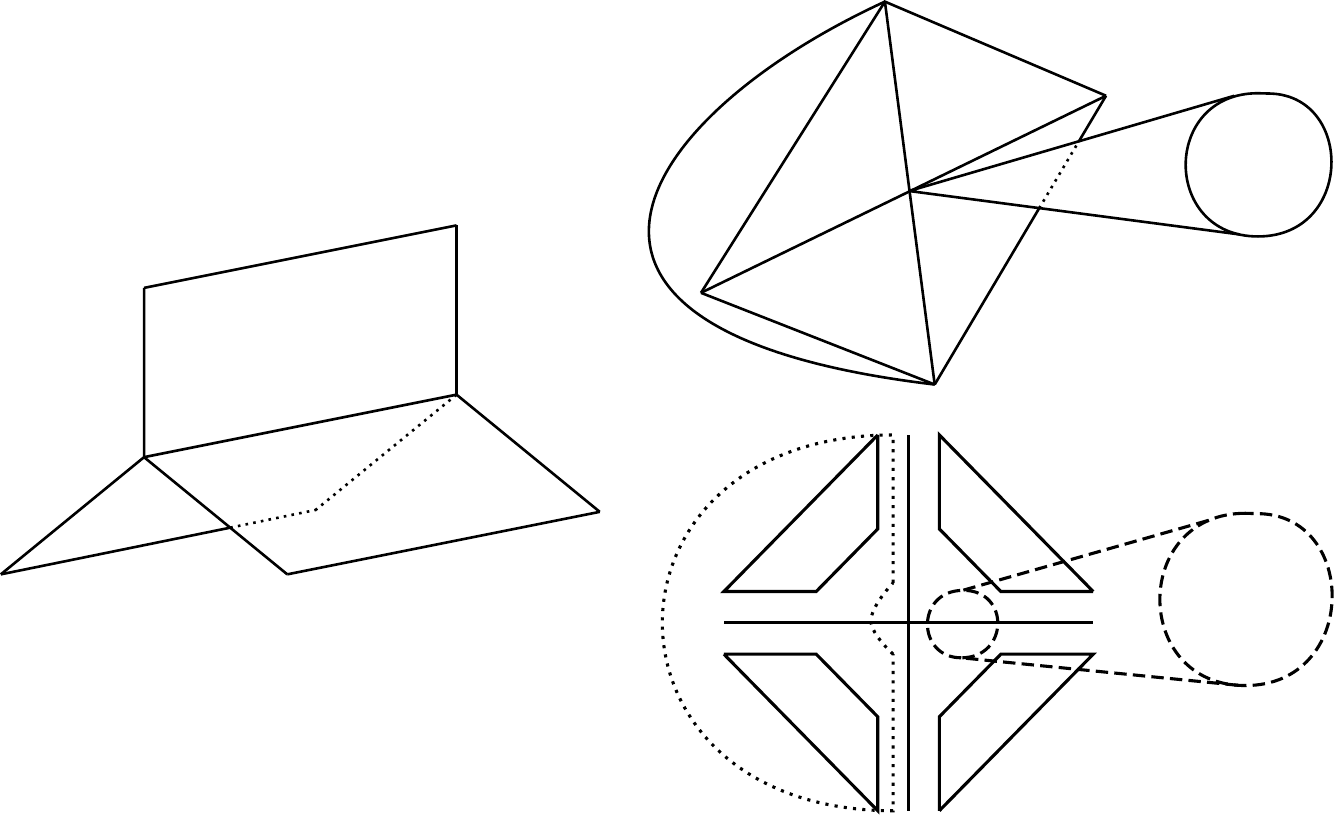
  \caption{%
    A typical $1$-block and a typical $0$-block. Below the $0$-block is its
    gluing pattern as described in the standard form
    \label{fig:Blocks}
  }
\end{figure}

\subsection{Maps of Surfaces to Generalized $2$-Complexes}
Now that we have a model for a neighborhood of the $1$-skeleton, we want to
lift this along a map $\Surface\to\GeometricRealization{\GeneralizedComplex}$.
This will require the following definition.
\begin{definition}[Transverse to the Membrane]
  Let $\GeneralizedComplex$ denote a combinatorial generalized $2$-complex.
  We call a map $\ContinuousMap \colon \Surface\to
  \GeometricRealization{\GeneralizedComplex}$
  \introduce{transverse to the membranes} if
  \begin{enumerate}[(a)]
    \item
      $\ContinuousMap$ is smooth when restricted to the preimage of an open
      $2$-cell $\Cell$ and $\ContinuousMap$ is transverse to
      $\IncludedMembrane{\Cell}$
    \item
    \label{itm:Transverse1Membrane}
      The map
      \[
        \SingularSetProjection{1}
        \circ
        \at{\ContinuousMap}{
          \apply{
            \ContinuousMap^{-1}
          }
        {
          \apply{
            \SingularSet{1}
          }
          {
            \GeneralizedComplex
          }
        }
        }
      \]
      is transverse  to $\IncludedMembrane{\Cell}$ for any
      $
        \Cell
        \in
        \CellsOfDim{1}{\GeneralizedComplex}
      $.
  \end{enumerate}
\end{definition}

\begin{lemma}
\label{lem:HomotopicToTransverse}
  Let $\GeneralizedComplex$ denote a combinatorial generalized $2$-complex,
  then every map
  $\ContinuousMap\colon \Surface\to \GeneralizedComplex$
  is homotopic to a map that is transverse to the membranes.
\end{lemma}
This lemma can be proven by first replacing the restriction of the map to the
preimage of a neighborhood of the plasma by a homotopic map that is
transverse to the membranes of the $2$-cells and the preimage of the membranes
of the $1$-cells under $\SingularSetProjection{1}$ intersected with the
membrane of the $2$-cell.
Secondly, using the block structure one can now replace the restriction of the
map to the preimage of the singular set by a map that is transverse to the
preimage of the membranes of $1$-cells under $\SingularSetProjection{1}$ and
that is homotopic relative to the boundary of the singular set to the original
restriction. Gluing these two maps together yields a homotopic map that is
transverse to the membranes. Because this uses the block structure we need the
generalized $2$-complex to be combinatorial.

In order to prove Theorem~\ref{thm:NormalForm} we will consecutively manipulate
maps in order to turn them into a cellwise covering without folds. In the first
step we will tighten the map to remove unnecessary folds in the singular set.
This is captured by the following definition.
\begin{definition}[No Backtracking and Allowable]%
  Let
  $
    \ContinuousMap
    \colon
    \Surface
    \to
    \GeometricRealization{\GeneralizedComplex}
  $
  denote a map that is transverse to the membranes.
  Let $\Surface_{\partial}$ denote the preimage of the union of all boundaries
  of blocks (which is an embedded graph in $\Surface$).
  We say that
  \introduce{$\ContinuousMap$ has no backtracking}
  if $\ContinuousMap$ is locally injective on $\Surface_{\partial}$.

  Let $X$ denote a generalized $2$-complex. We say that a map $\Surface \to
  \GeometricRealization{\GeneralizedComplex}$
  is \introduce{allowable} if it is transverse to the membranes and
  the restriction to all preimages of plasma of $2$-cells is allowable.
\end{definition}

After replacing a map with an allowable one without backtracking, we will use
the main result of \cite{SkoraDegree} for surfaces with
boundary to turn the map into a branched covering on the preimage of open
$2$-cells. By pushing the branching points into the $1$-skeleton, we will get
an honest covering.
Then we deal with $1$-cells
by moving potential singularities into the $0$-cells.
This proof strategy is captured in the following lemma:
\begin{lemma}
  \label{lem:HomotopicCellwiseCovering}
  Let $\GeneralizedComplex$ denote a combinatorial generalized
  $2$-complex and $\ContinuousMap \colon \Surface \to
  \TopologicalSpace$ a continuous map which admits no squeezes.
  Then $\ContinuousMap$ is equivalent in the compression order to
  $\tilde{\ContinuousMap}$ such
  that $\tilde{\ContinuousMap}$ satisfies everything below:
  \begin{enumerate}[(a)]
    \item
    \label{itm:TransverseMembrane}
      $\tilde{\ContinuousMap}$ is transverse to the membranes
    \item
      \label{itm:NoFoldsAllowable} $\tilde{\ContinuousMap}$
      admits no backtracking and is allowable
    \item
      \label{itm:PlasmaBranchedCovering} the restriction of
      $\tilde{\ContinuousMap}$ to the preimage of the
      plasma of any $2$-cell is a branched covering
    \item
      \label{itm:PlasmaCovering}the restriction of
      $\tilde{\ContinuousMap}$ to the preimage of the
      plasma of any $2$-cell is a covering
    \item
      \label{itm:1Blocks} The preimage of the $1$-block
      belonging to any $\Cell$ consists of rectangles which
      can be parameterized such that $\tilde{\ContinuousMap}$
      is the inclusion of
      \[
        \raisebox{.2em}{$
          \Interval
          \times
          \Interval
          \times
          \{l,m\}
        $}
        /
        \raisebox{-.2em}{$
          \left(
            0,s,m
          \right)
          \sim
          \left(0,s,l\right)
        $}
        \to
        \raisebox{.2em}{$
          \Interval
          \times
          \Interval
          \times
          \{1,\ldots,n\}
        $}
        /
        \raisebox{-.2em}{$
          \left(
          0,s,i
          \right)
          \sim
          \left(
          0,s,j
          \right)
        $}
      \]
      into the standard form of the $1$-block belonging to $\Cell$,
      where
      $
        \{
          1,\ldots, n_{\Cell}
        \}
        \ni
        l
        \neq
        m
        \in
        \{
          1,\ldots,n_{\Cell}
        \}
      $
    \item
      \label{itm:0Blocks} a connected component of the preimage of
      any $0$-block can either
      be parameterized as
      \[
        \raisebox{.2em}{$
         \Circle \times \Interval
        $}
        /
        \raisebox{-.2em}{$
          \left(0,z\right)
          \sim
          \left(0,z'\right)
        $}
      \]
      or it can be parameterized as
      \[
        \left.
        \raisebox{.2em}
        {$
            \Interval
            \times
            \Interval
            \times
            \mathbb{Z}/k\mathbb{Z}
            \cup
            \Interval
            \times
            \mathbb{Z}/k\mathbb{Z}
        $}
        \Big/
        \raisebox{-0.75cm}{$
          \begin{gathered}
            \left(0,x,i\right)
            =
            \left(0,y,j\right)
            =
            \left(0,l\right)
            \\
            \left(x,0,i\right)
            =
            \left(x,i\right)
            \\
            \left(x,1,i\right)
            =
            \left(x,i+1\right)
          \end{gathered}
        $}
        \right.
      \]
      In this parameterization $\tilde{\ContinuousMap}$ is
      componentwise the inclusion with respect to the standard form
      of the $0$-block.
  \end{enumerate}
\end{lemma}

In order to prove (\ref{itm:NoFoldsAllowable}) in
Lemma~\ref{lem:HomotopicCellwiseCovering} we will need to measure how folded a
map is. This is encapsulated in the following definition.
\begin{definition}[Folding Degree $\Degree$]%
\label{dfn:Degree}
  Let $\GeneralizedComplex$ be a combinatorial generalized $2$-complex,
  and consider the set of edges $\BoundaryGraphEdges{\GeneralizedComplex}$
  of the block boundary graph,
  each element being an interval or circle in $\GeneralizedComplex$. We can
  decompose $\BoundaryGraphEdges{\GeneralizedComplex}$ into the set of edges
  $
    \BoundaryGraphEdges{\GeneralizedComplex}^{\ContinuousMap}_{\infty}
  $
  where there exists a point in the edge which has an infinite preimage and its
  complement $\BoundaryGraphEdges{\GeneralizedComplex}^{\ContinuousMap}_{fin}$
  We can define
  $
    \BoundaryGraphVertices{\GeneralizedComplex}^{\ContinuousMap}_{\infty}
  $
  and
  $
    \BoundaryGraphVertices{\GeneralizedComplex}^{\ContinuousMap}_{fin}
  $
  analogously.
  $\BoundaryGraphVertices{\GeneralizedComplex}$.
  We define the \introduce{folding degree} $\apply{\Degree}{\ContinuousMap}$
  of a map
  $
    \ContinuousMap
    \colon
    \Surface
    \to
    \GeometricRealization{\GeneralizedComplex}
  $:
  \begin{align*}
    \apply{\Degree}{\ContinuousMap}
    \coloneqq &
    \left(
      \Cardinality{
          \BoundaryGraphEdges{\GeneralizedComplex}^{\ContinuousMap}_{\infty}
      }
      +
      \Cardinality{
        \BoundaryGraphVertices{\GeneralizedComplex}^{\ContinuousMap}_{\infty}
      }
      ,
      \sum_{%
        C \in \BoundaryGraphEdges{\GeneralizedComplex}^{\ContinuousMap}_{fin}
      }
      \sup
      \left\{
        \Cardinality{%
          \apply{\ContinuousMap^{-1}}{p}
        }
        \middle|
        p \in C \setminus \BoundaryGraphVertices{\GeneralizedComplex}
      \right\}
      +
      \sum_{%
        p \in \BoundaryGraphVertices{\GeneralizedComplex}^{\ContinuousMap}_{fin}
      }
      \Cardinality{%
        \apply{\ContinuousMap^{-1}}{p}
      }
    \right)
  \end{align*}
  We denote the first coordinate by $\DegreeInfty$ and the second coordinate by
  $\DegreeFinite$.
\end{definition}

\begin{lemma}
  \label{lem:dMinimizerNoBacktracking}
  Let $\GeneralizedComplex$ be a generalized $2$-complex and
  $
    \ContinuousMap
    \colon
    \Surface
    \to
    \GeometricRealization{\GeneralizedComplex}
  $
  be a map from a closed oriented surface
  which is transverse to the membranes, admits no squeezes,
  and has minimal $\apply{\Degree}{\ContinuousMap}$ in its equivalence class of
  the compression order (Here the minimum is taken in the lexicographic order).
  Then $\ContinuousMap$ has no backtracking
  and $\apply{\Degree_{\infty}}{\ContinuousMap} = 0$.
\end{lemma}
\begin{proof}
  We will first show that if $\Degree$ is minimal, then $\ContinuousMap$ has to
  be locally injective on the preimage of membranes of $1$-cells and the
  adjacent edges in the block boundary graph.
  Then we will proceed with the membranes of $2$-cells.
  We will do this by showing that if $\ContinuousMap$ was not already of a
  particular form, then $\apply{\Degree}{\ContinuousMap}$ could
  not have been minimal.

  Let $S$ denote a star in the block boundary graph consisting of a point in
  the membrane of a $1$-cell and all its adjacent edges.
  The preimage of $S$ under $\ContinuousMap$ consists of a collection of
  arcs and circles. Let us first concentrate on the arcs and then on the
  circles.

  By transversality, $\ContinuousMap$ is a diffeomorphism in a
  neighborhood of the boundary of these arcs.
  Pick an arc in the preimage and denote it by $a'$, and
  fix a tubular neighborhood $\apply{N}{S}$ of $S$ such that its preimage
  is a tubular neighborhood of the preimage of $S$.
  Let us denote the restriction of this tubular neighborhood to $a'$ by
  $\apply{N}{a'}$.
  We have to distinguish two cases: Either $a'$ connects different boundary
  points of the star or it maps both endpoints to the same boundary point.
  
  Relative to two endpoints of the star there
  exists a deformation of the star, linear on the rays, which deforms the star
  to the union of the two edges connecting the two boundary points.
  Using the extra coordinate of the tubular neighborhood, we can extend this
  homotopy to $\apply{N}{S}$ such that it is the identity on the part of
  boundary of the tubular neighborhood that is not a tubular neighborhood of
  the boundary of the star. Postcomposing
  $
    \at{\ContinuousMap}
      {\apply{N}{a'}}
  $
  with the
  homotopy corresponding to the image of the endpoints of $a'$ yields a new
  map such that the image of $a'$ is an interval. Using an extension of the
  linear homotopy between $\at{\ContinuousMap}{a'}$
  and the unique affine map with the same endpoints as
  $\at{\ContinuousMap}{a'}$ to $\apply{N}{a'}$ yields a new map, which is
  affine on $a'$.
  
  Let us now assume that $a'$ maps both its endpoints to the same boundary
  point.
  We can again use the deformation of the star relative to the boundary
  point in the image and an arbitrary second boundary point described before
  (if there is no second boundary point the argument works without the
  deformation) so that the image of $a'$ is an interval.
  Since the image of the arc is contained in two rays (or a single ray) we can
  treat it as lying in an interval. Postcomposing the map on $\apply{N}{a'}$
  with the isotopy depicted in Figure~\ref{fig:IsotopyDMinimizer} yields a map
  that maps $a'$ to the plasma of a $2$-cell. The resulting map is still
  transverse to the membrane since it was transverse to the preimage of the
  membranes under the isotopy.
  \begin{figure}[h]%
\begingroup%
  \makeatletter%
  \providecommand\color[2][]{%
    \errmessage{(Inkscape) Color is used for the text in Inkscape, but the package 'color.sty' is not loaded}%
    \renewcommand\color[2][]{}%
  }%
  \providecommand\transparent[1]{%
    \errmessage{(Inkscape) Transparency is used (non-zero) for the text in Inkscape, but the package 'transparent.sty' is not loaded}%
    \renewcommand\transparent[1]{}%
  }%
  \providecommand\rotatebox[2]{#2}%
  \newcommand*\fsize{\dimexpr\f@size pt\relax}%
  \newcommand*\lineheight[1]{\fontsize{\fsize}{#1\fsize}\selectfont}%
  \ifx\svgwidth\undefined%
    \setlength{\unitlength}{261.26867177bp}%
    \ifx\svgscale\undefined%
      \relax%
    \else%
      \setlength{\unitlength}{\unitlength * \real{\svgscale}}%
    \fi%
  \else%
    \setlength{\unitlength}{\svgwidth}%
  \fi%
  \global\let\svgwidth\undefined%
  \global\let\svgscale\undefined%
  \makeatother%
  \begin{picture}(1,0.77649953)%
    \lineheight{1}%
    \setlength\tabcolsep{0pt}%
    \put(0,0){\includegraphics[width=\unitlength,page=1]{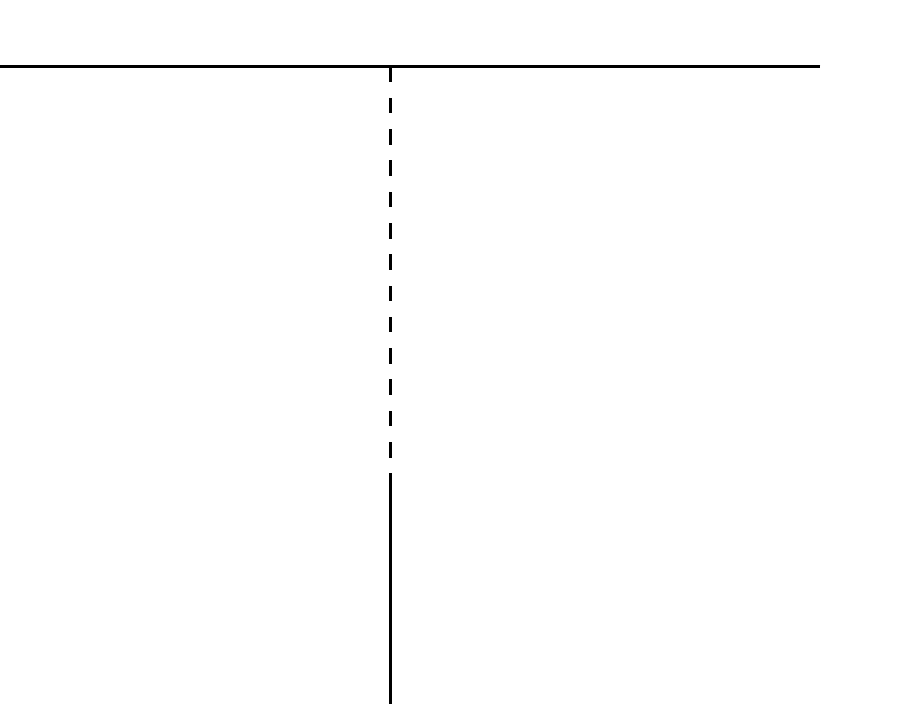}}%
    \put(0.39461193,0.71740225){\color[rgb]{0,0,0}\makebox(0,0)[lt]{\lineheight{1.25000012}\smash{\begin{tabular}[t]{l}$\apply{\ContinuousMap}{a'}$\end{tabular}}}}%
    \put(0,0){\includegraphics[width=\unitlength,page=2]{Figures/IsotopyDMinimizer.pdf}}%
  \end{picture}%
\endgroup%

    \caption{%
      The isotopy pushes the dashed line to the part of the membrane between
      its endpoints. This is only defined on the interval and the indicated
      neighborhood not on the whole star.
      \label{fig:IsotopyDMinimizer}
    }
  \end{figure}

  Doing this for all arcs in the preimage of all stars yields
  a new map which is locally injective on every arc in the preimage of the
  stars.
  Let us denote this map by $\ContinuousMap'$.

  So suppose that there exists a circle in the preimage of a star.
  Since $\ContinuousMap$ admits no squeeze this circle necessarily bounds a
  disk.
  If the restriction of $\ContinuousMap'$ to this disk is homotopic
  relative its boundary to a constant map, we can use this homotopy to remove
  all preimages of the block boundary graph in this disk.
  If the restriction of $\ContinuousMap$ does not admit such a homotopy we can
  perform a compression at its boundary and then assume that it has such a
  homotopy.
  Doing this for all circles yields a new map which is locally injective on the
  preimage of all stars, is transverse to the membranes and
  equivalent to $\ContinuousMap$.
  If $\ContinuousMap$ was not already of this form, then
  $\apply{\Degree}{\ContinuousMap'}<\apply{\Degree}{\ContinuousMap}$.
  Therefore we conclude that $\ContinuousMap$ was already locally injective on
  the preimage of stars.

  Let us now focus on membranes of $2$-cells. The arguments that follow are
  almost identical to the previous ones.
  As was noted before, since $\ContinuousMap$
  is transverse to the membranes, it is a diffeomorphism in a neighborhood of
  a preimage of a vertex of the block boundary graph in the membranes of
  $2$-cells.
  
  Consider a single edge $a$ in the block boundary graph of
  $\GeneralizedComplex$ lying in the membrane of a $2$-cell.
  By transversality the preimage of $a$ consists of arcs and circles.
  Again let  us concentrate on arcs first and then deal with circles.
  Let us pick an arc in the preimage and denote it by $a'$.
  Fix a tubular neighborhood $\apply{N}{a}$ of $a$ such that its preimage
  is a tubular neighborhood of the preimage.
  Let us denote the restriction of this tubular neighborhood to $a'$ by
  $\apply{N}{a'}$.
  
  If $a'$ connects different vertices in the block boundary graph, then
  using an extension of the linear homotopy between $\at{\ContinuousMap}{a'}$
  and the unique affine map with the same endpoints as
  $\at{\ContinuousMap}{a'}$ to $\apply{N}{a'}$ yields a new map, which is
  affine on this arcs in the preimage. Otherwise use an isotopy on the membrane
  similar to the one depicted in Figure~\ref{fig:IsotopyDMinimizer} to map $a'$
  completely into the image of its adjacent edge in the block boundary graph.
  Note that by transversality the neighboring edges have to map to the same
  edge in the block boundary graph.

  Doing this for all arcs in the preimage of the membranes of $2$-cells yields
  a new map. After finitely many repetitions,
  this process ends with a map that is locally
  injective except for circle components that map to the inner of edges,
  and we can remove those
  up to equivalence in the same way we did for circles in the preimage of
  stars.
  This yields a new map which is locally injective on the
  preimage of all stars, is transverse to the membranes and
  equivalent to $\ContinuousMap$.
  If $\ContinuousMap$ was not already of this form, then
  $\apply{\Degree}{\ContinuousMap'}<\apply{\Degree}{\ContinuousMap}$.
  Therefore we conclude that $\ContinuousMap$ was already locally injective on
  the preimage of membranes of $2$-cells.

  By compactness of $\Surface$, every map without backtracking $\ContinuousMap$
  has $\apply{\DegreeInfty}{\ContinuousMap}=0$.
\end{proof}

\begin{lemma}%
\label{lem:dMinimizerAllowable}
  Let $\GeneralizedComplex$ be a generalized $2$-complex and
  $
    \ContinuousMap
    \colon
    \Surface
    \to
    \GeometricRealization{\GeneralizedComplex}
  $
  be a map from a closed oriented surface
  which is transverse to the membranes, admits no squeezes,
  and has minimal $\apply{\Degree}{\ContinuousMap}$ in its equivalence class of
  the compression order (Here the minimum is taken in the lexicographic order).
  Then $\ContinuousMap$ is allowable.
\end{lemma}
\begin{proof}
  By Lemma~\ref{lem:dMinimizerNoBacktracking},
  $\ContinuousMap$ has no backtracking.
  Therefore $\ContinuousMap$ is a covering on the preimage of the membranes
  of $2$-cells.

  To prove that $\ContinuousMap$ is allowable,
  assume that $\ContinuousMap$ is not allowable.
  Then there is a $2$-cell $\Cell$, such that $\ContinuousMap$ is not allowable
  over $\IncludedPlasma{\Cell}$.
  We abbreviate
  $
    \at{\ContinuousMap}{\apply{\ContinuousMap^{-1}}{\IncludedPlasma{\Cell}}}
  $
  by $\ContinuousMap_{\Cell}$.
  This means that there is a component $c$ of $\IncludedMembrane{\Cell}$
  such that the degree of
  $
    \at{\ContinuousMap_{\Cell}}{\apply{\ContinuousMap^{-1}}{c}}
    \colon
    \apply{\ContinuousMap^{-1}}{c}
    \to
    c
  $
  is greater than $\apply{\GeometricDegree}{\ContinuousMap_{\Cell}}$.

  By the definition of geometric degree,
  there is a map $\ContinuousMap'_{\Cell}$
  homotopic to $\ContinuousMap_{\Cell}$, still mapping to
  $\IncludedPlasma{\Cell}$, which is a
  $
    \apply{\GeometricDegree}{\ContinuousMap_{\Cell}}
    =
    \apply{\GeometricDegree}{\ContinuousMap'_{\Cell}}
  $%
  -fold covering over an open disk
  $
    \Ball{2}
    \subset
    \IncludedPlasma{\Cell}
  $.
  Let $\ContinuousMap'$ be defined by extending $\ContinuousMap'_{\Cell}$ by
  $\ContinuousMap$. Since the homotopy between $\ContinuousMap$ and
  $\ContinuousMap'$ is supported in the inner of the plasma of $\Cell$, we
  conclude that
  $\apply{\Degree}{\ContinuousMap}=\apply{\Degree}{\ContinuousMap'}$.
  Fix an embedded path $y\subset \IncludedCell{\Cell}$
  starting at a point $\Point \in \IncludedMembrane{\Cell}\cap
  \BoundaryGraphVertices{\GeneralizedComplex}$
  and ending in $\partial \Ball{2}$,
  transverse to $\ContinuousMap'_{\Cell}$.
  We will use $y$ to reduce $\Degree$ by merging two boundary
  components of $\apply{\ContinuousMap'^{-1}}{\IncludedPlasma{\Cell}}$ over
  $y$.

  The preimage of $y$ consists of a collection of arcs with
  one or two endpoints in
  $
  \apply{
    \ContinuousMap^{\prime-1}
  }{
    \IncludedMembrane{\Cell}
  }
  $.
  Since the degree of
  $
    \at{
      \ContinuousMap'_{\Cell}
    }{
      \apply{\ContinuousMap^{\prime -1}}{c}
    }
    \colon
    \apply{\ContinuousMap^{\prime -1}}{c}
    \to
    c
  $
  is greater than $\apply{\GeometricDegree}{\ContinuousMap'_{\Cell}}$,
  there exists an arc $y'$ in the preimage of $y$ that connects two points in
  $\apply{\ContinuousMap_{\Cell}^{\prime -1}}{c}$.
  By orientability of $\Surface$, $y'$ connects two different components of
  $
    \apply{
      \ContinuousMap^{\prime -1}_{\Cell}
    }{
      \IncludedMembrane{\Cell}
    }
  $.

  Fix a small tubular neighborhood $\apply{N}{y}$ of $y$ such that
  $
    \apply{
      \ContinuousMap_{\Cell}^{-1}
    }{
      \apply{N}{y}
    }
  $
  is a tubular neighborhood of $y'$ as well. Let us denote this tubular
  neighborhood by $\apply{N}{y'}$.
  Note that by counting preimages
  $
    \apply{
      \ContinuousMap
    }{
      y'
    }
  $
  cannot intersect
  $\Ball{2}$.
  Let $\ContinuousMap''$ be defined as $\ContinuousMap'$ outside of
  $\apply{N}{y'}$ and in $\apply{N}{y'}$ be defined as $\ContinuousMap'$
  postcomposed with the
  diffeomorphism depicted in Figure~\ref{fig:IsotopyAllowable}. Evidently
  $\ContinuousMap''$ is homotopic to $\ContinuousMap'$ furthermore note that
  $\ContinuousMap''$ is still transverse to the membranes since the vertical
  tangent vectors (in Figure~\ref{fig:IsotopyAllowable}) lie in the image of
  the differential of $\ContinuousMap'$ and the only points where these do not
  suffice to have transversality to the preimage of the membrane under the
  diffeomorphism in Figure~\ref{fig:IsotopyAllowable} lie in $\Ball{2}$ and
  hence do not have a preimage in $\apply{N}{y'}$.

  By construction
  $
    \apply{
      \ContinuousMap'^{-1}
    }{
      \Point
    }
  $
  has two more points than
  $
    \apply{
      \ContinuousMap'^{-1}
    }{
      \Point
    }
  $
  and the other summands of
  $\apply{\DegreeFinite}{\ContinuousMap''}$ agree with the ones in
  $\apply{\DegreeFinite}{\ContinuousMap}$.
  \begin{figure}[h]%
    \def\svgwidth{1\textwidth}
    \hspace*{25pt}
\begingroup%
  \makeatletter%
  \providecommand\color[2][]{%
    \errmessage{(Inkscape) Color is used for the text in Inkscape, but the package 'color.sty' is not loaded}%
    \renewcommand\color[2][]{}%
  }%
  \providecommand\transparent[1]{%
    \errmessage{(Inkscape) Transparency is used (non-zero) for the text in Inkscape, but the package 'transparent.sty' is not loaded}%
    \renewcommand\transparent[1]{}%
  }%
  \providecommand\rotatebox[2]{#2}%
  \newcommand*\fsize{\dimexpr\f@size pt\relax}%
  \newcommand*\lineheight[1]{\fontsize{\fsize}{#1\fsize}\selectfont}%
  \ifx\svgwidth\undefined%
    \setlength{\unitlength}{556.2001842bp}%
    \ifx\svgscale\undefined%
      \relax%
    \else%
      \setlength{\unitlength}{\unitlength * \real{\svgscale}}%
    \fi%
  \else%
    \setlength{\unitlength}{\svgwidth}%
  \fi%
  \global\let\svgwidth\undefined%
  \global\let\svgscale\undefined%
  \makeatother%
  \begin{picture}(1,0.32669989)%
    \lineheight{1}%
    \setlength\tabcolsep{0pt}%
    \put(0,0){\includegraphics[width=\unitlength,page=1]{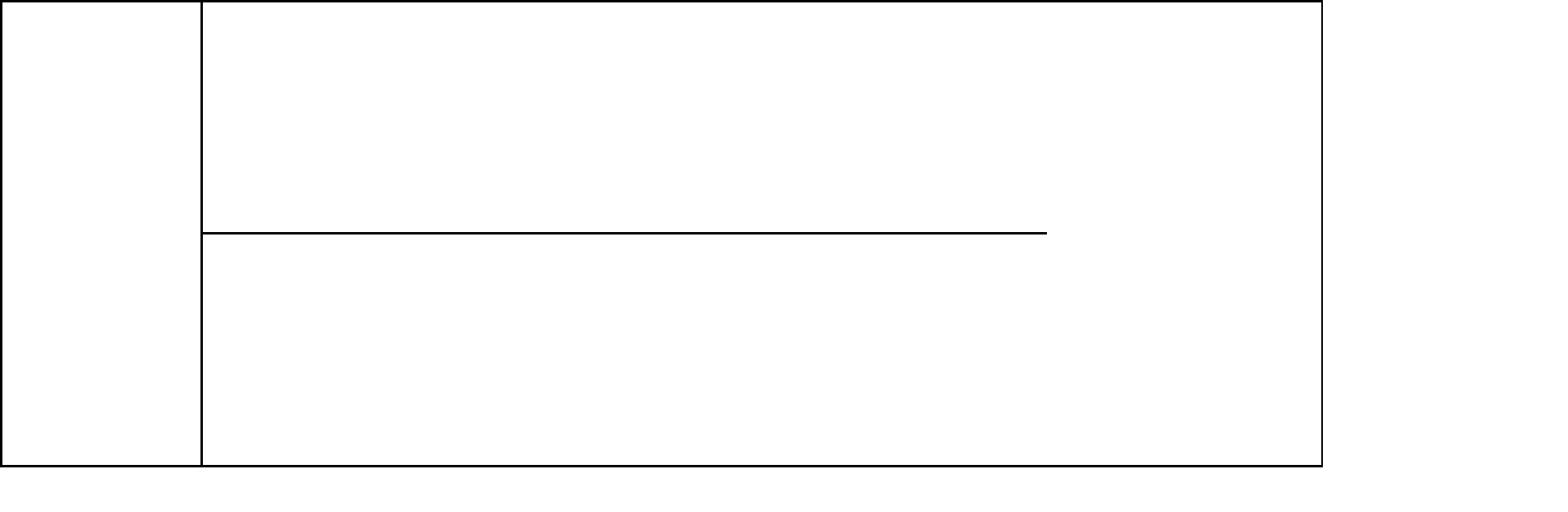}}%
    \put(0.85693067,0.25860391){\color[rgb]{0,0,0}\makebox(0,0)[lt]{\begin{minipage}{0.19396424\unitlength}\raggedright $N\left(y\right)$\end{minipage}}}%
    \put(0.10344758,0.01883926){\color[rgb]{0,0,0}\makebox(0,0)[lt]{\begin{minipage}{0.35059319\unitlength}\raggedright $\IncludedMembrane{\Cell}$\end{minipage}}}%
    \put(0,0){\includegraphics[width=\unitlength,page=2]{Figures/IsotopyAllowable.pdf}}%
    \put(0.68163408,0.30579915){\color[rgb]{0,0,0}\makebox(0,0)[lt]{\lineheight{1.25000012}\smash{\begin{tabular}[t]{l}$\Ball{2}$\end{tabular}}}}%
    \put(0,0){\includegraphics[width=\unitlength,page=3]{Figures/IsotopyAllowable.pdf}}%
    \put(0.67369848,0.17259633){\color[rgb]{0,0,0}\makebox(0,0)[lt]{\lineheight{1.25000012}\smash{\begin{tabular}[t]{l}$y$\end{tabular}}}}%
    \put(0,0){\includegraphics[width=\unitlength,page=4]{Figures/IsotopyAllowable.pdf}}%
  \end{picture}%
\endgroup%

    \caption{%
      The diffeomorphism is supported in $\apply{N}{y}$ and pushes the dashed
      portion
      of $D$ between the endpoints of the dotted line along $y$ to the dotted
      line. The other dashed line is the preimage of the portion of the plasma
      between the dotted line under this diffeomorphism.
      \label{fig:IsotopyAllowable}
    }
  \end{figure}
\end{proof}
\begin{proof}[Proof of Lemma~\ref{lem:HomotopicCellwiseCovering}]
  Since representatives in the equivalence class of $\ContinuousMap$ differ by
  compression at an inessential curve, there always exists a representative
  that minimizes $\Degree$ and admits no squeezes if $\ContinuousMap$ admitted
  no squeezes to begin with.
  Lemma~\ref{lem:HomotopicToTransverse} implies that such a representative is
  equivalent to a map that is transverse to the membranes which minimizes
  $\Degree$ in its equivalence class,
  hence satisfies Part~(\ref{itm:TransverseMembrane})
  Lemma~\ref{lem:dMinimizerNoBacktracking}
  and
  Lemma~\ref{lem:dMinimizerAllowable} imply that this
  map also satisfies Part~(\ref{itm:NoFoldsAllowable}).

  So let us assume that $\ContinuousMap$ fulfills (\ref{itm:TransverseMembrane})
  and (\ref{itm:NoFoldsAllowable}), then Theorem~1.1 in
  \cite{SkoraDegree} implies that there is a homotopy of the restriction of
  $\ContinuousMap$ to the preimage of all plasma of $2$-cells relative to the
  boundary of this preimage such that the resulting map is a branched covering
  over the plasma of $2$-cells.
  Hence for this map Part~(\ref{itm:TransverseMembrane}) to
  Part~(\ref{itm:PlasmaBranchedCovering}) hold.

  Now let us assume that $\ContinuousMap$ fulfills
  Part~(\ref{itm:TransverseMembrane}),~(\ref{itm:NoFoldsAllowable})
  and~(\ref{itm:PlasmaBranchedCovering}).
  Pick an embedded path from a branching point in the plasma
  of a cell to the intersection of the $1$-singular set with that
  cell.
  Postcompose $\ContinuousMap$ with an isotopy, supported in
  a tubular neighborhood of this path, that pushes along this path.
  By repeating this for every other branching point, we
  obtain a homotopic map that satisfies
  Parts~(\ref{itm:TransverseMembrane})-(\ref{itm:PlasmaCovering}).

  Let us assume that $\ContinuousMap$ already
  fulfills~(\ref{itm:TransverseMembrane})
  to~(\ref{itm:PlasmaCovering}).
  In order to find a map fulfilling~(\ref{itm:TransverseMembrane})
  to~(\ref{itm:1Blocks}), we will first alter $\ContinuousMap$ so that the
  preimage of $1$-blocks consists of rectangles and so that the previous parts
  of the lemma still hold.
  Then we will alter the resulting map on these rectangles.

  Let $\Cell$ denote a $1$-cell.
  Since $\ContinuousMap$ is transverse to the membranes, a component
  of the preimage of the $1$-block $\Block{\Cell}$ is bounded by embedded
  closed polygons.
  Since $\ContinuousMap$ admits no squeezes, the embedded closed polygons
  bound disks and each such component has genus $0$.
  Then we can perform surgery at each embedded closed polygon using the
  nullhomotopy collapsing it in the block.
  Since the curve bounded a disk this does not change the equivalence class in
  the compression order and by construction this did not change $\Degree$.
  Then the preimage of the $1$-block consists of disks and spheres.
  Removing these spheres does not alter the equivalence class of
  $\ContinuousMap$.
  Hence we can assume without loss of generality that the preimage of the
  $1$-blocks consists of disks.
  Furthermore since $\ContinuousMap$ has no backtracking, the edges of the
  polygon map alternately to
  $\apply{\SingularSetProjection{1}^{-1}}{\IncludedMembrane{\Cell}}$
  and
  $
    \bigcup_{
      \Cell'
      \in
      \CellsOfDim{\GeneralizedComplex}{2}
    }
    \left(
      \IncludedMembrane{\Cell'}
      \cap
      \Block{\Cell}
    \right)
  $.
  Furthermore, if we number the edges cyclically then every edge labeled $i$
  of the polygon maps to a different edge of $\partial \Block{\Cell}$ than the
  edges la belled $i-2$ and $i+2$.

  Let us decompose $\IncludedMembrane{\Cell}$ into the point
  $\IncludedMembrane{\Cell}^{0}$ and the other point
  $\IncludedMembrane{\Cell}^{1}$.
  Since $\ContinuousMap$ admits no backtracking we also know that the edges of
  the polygon mapping to
  $\apply{\SingularSetProjection{1}^{-1}}{\IncludedMembrane{\Cell}}$
  map alternately to
  $\apply{\SingularSetProjection{1}^{-1}}{\IncludedMembrane{\Cell}^{0}}$ and
  $\apply{\SingularSetProjection{1}^{-1}}{\IncludedMembrane{\Cell}^{1}}$.

  Connect the middle points of all edges in
  $
    \apply{
      \ContinuousMap^{-1}
    }
    {
      \apply{
        \SingularSetProjection{1}^{-1}
      }
      {
         \IncludedMembrane{\Cell}^{0}
      }
    }
  $
  to the center point of the polygon. This yields an embedded star $S$ in the
  polygon and we can find a map $\ContinuousMap'$ homotopic to $\ContinuousMap$
  which agrees with $\ContinuousMap$ outside the inner of the polygon and such
  that $\ContinuousMap'$ maps the star to $\IncludedMembrane{\Cell}^{0}$.
  Hence we obtain the following commutative diagram:
  \[
    \begin{tikzcd}
      \Surface
      \ar[d]
      \ar[r,"\ContinuousMap'"]
      &\GeneralizedComplex
      \\
      \Surface/S
      \ar[ru]
    \end{tikzcd}
  \]
  The left map is a homotopy equivalence. Let us denote the diagonal
  map by $\tilde{\ContinuousMap}'$.
  The map $\tilde{\ContinuousMap}'$ is in general not transverse to the
  membranes anymore. Let us fix this: Now the preimage of $\Block{\Cell}$
  consists of rectangles that possibly intersect in a single point in their
  boundary, and the restriction of $\tilde{\ContinuousMap}'$ to the boundary of
  such a rectangle is a homeomorphism. Therefore we can replace
  $\tilde{\ContinuousMap}'$ by a homotopic map that has the standard form on
  these rectangles.

  Postcomposing this map with an isotopy as depicted on the left side of
  Figure~\ref{fig:Isotopy1Block} yields a new map that we denote
  $\bar{\ContinuousMap}'$, which is again transverse to the membranes.
  This procedure does not change $\Degree$ and therefore
  Part~(\ref{itm:TransverseMembrane})
  and~(\ref{itm:NoFoldsAllowable}) hold by
  Lemma~\ref{lem:dMinimizerNoBacktracking} and
  Lemma~\ref{lem:dMinimizerAllowable} for
  $\bar{\ContinuousMap}'$ as well.
  Since $\ContinuousMap$ agrees with $\bar{\ContinuousMap}'$ over the plasma
  of the $2$-cells, Part~(\ref{itm:PlasmaBranchedCovering})
  and~(\ref{itm:PlasmaCovering}) are still true for $\bar{\ContinuousMap}'$.
  The preimage of $\Block{\Cell}$ consists of
  disjoint rectangles as depicted in Figure~\ref{fig:Isotopy1Block}.
  Furthermore $\bar{\ContinuousMap}'$ embeds such a rectangle
  into $\Block{\Cell}$.
  Hence $\bar{\ContinuousMap}'$ fulfills Parts~(\ref{itm:TransverseMembrane})
  to~(\ref{itm:1Blocks}).

  \begin{figure}[h]%
    \def\svgwidth{1.3\textwidth}
\begingroup%
  \makeatletter%
  \providecommand\color[2][]{%
    \errmessage{(Inkscape) Color is used for the text in Inkscape, but the package 'color.sty' is not loaded}%
    \renewcommand\color[2][]{}%
  }%
  \providecommand\transparent[1]{%
    \errmessage{(Inkscape) Transparency is used (non-zero) for the text in Inkscape, but the package 'transparent.sty' is not loaded}%
    \renewcommand\transparent[1]{}%
  }%
  \providecommand\rotatebox[2]{#2}%
  \newcommand*\fsize{\dimexpr\f@size pt\relax}%
  \newcommand*\lineheight[1]{\fontsize{\fsize}{#1\fsize}\selectfont}%
  \ifx\svgwidth\undefined%
    \setlength{\unitlength}{603.29743171bp}%
    \ifx\svgscale\undefined%
      \relax%
    \else%
      \setlength{\unitlength}{\unitlength * \real{\svgscale}}%
    \fi%
  \else%
    \setlength{\unitlength}{\svgwidth}%
  \fi%
  \global\let\svgwidth\undefined%
  \global\let\svgscale\undefined%
  \makeatother%
  \begin{picture}(1,0.24535144)%
    \lineheight{1}%
    \setlength\tabcolsep{0pt}%
    \put(0,0){\includegraphics[width=\unitlength,page=1]{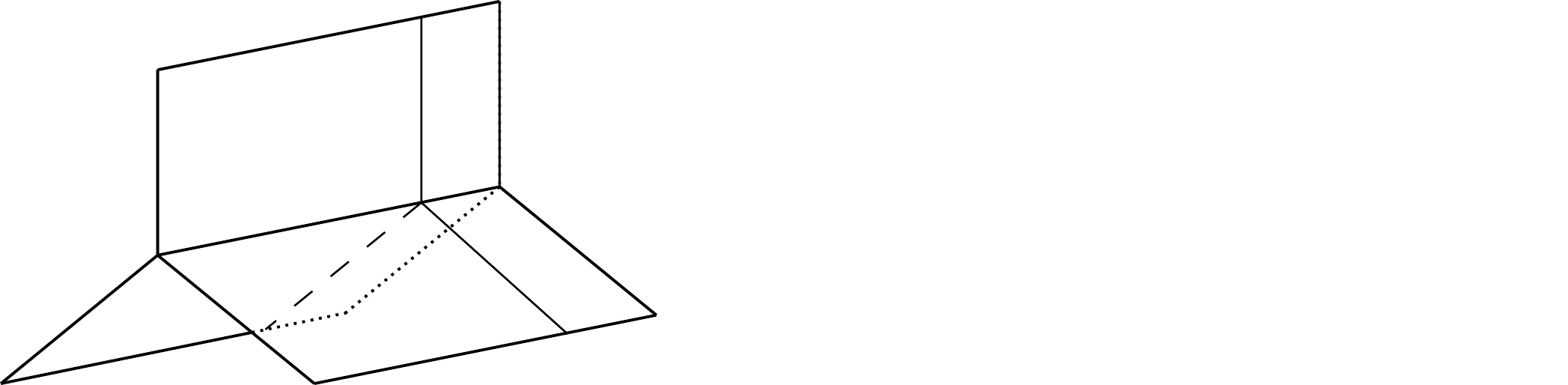}}%
    \put(0.19780749,0.19472961){\color[rgb]{0,0,0}\makebox(0,0)[lt]{\lineheight{1.25000012}\smash{\begin{tabular}[t]{l}$\apply{\SingularSetProjection{1}^{-1}}{\IncludedMembrane{\Cell}^{0}}$\end{tabular}}}}%
    \put(0,0){\includegraphics[width=\unitlength,page=2]{Figures/Isotopy1Block.pdf}}%
    \put(0.52519242,0.20513045){\color[rgb]{0,0,0}\makebox(0,0)[lt]{\lineheight{1.25000012}\smash{\begin{tabular}[t]{l}$\apply{\tilde{\ContinuousMap}'^{-1}}{\Block{\Cell}}$\end{tabular}}}}%
    \put(0,0){\includegraphics[width=\unitlength,page=3]{Figures/Isotopy1Block.pdf}}%
  \end{picture}%
\endgroup%

    \caption{%
      The isotopy pushes the indicated union of lines to
      $\apply{\SingularSetProjection{1}^{-1}}{\IncludedMembrane{\Cell}^{0}}$.
      The change to the preimage of the 1-Block is depicted on the right.
      \label{fig:Isotopy1Block}
    }
  \end{figure}

  So let us assume that $\ContinuousMap$ fulfills
  Parts~(\ref{itm:TransverseMembrane})-(\ref{itm:1Blocks}).
  Similar to the previous situation we can assume without loss of generality
  that the preimage of a $0$-block consists of polyhedra and disks.
  Since the $0$-block is homeomorphic to the cone of its boundary, and the
  preimage is homeomorphic to the componentwise cone of its boundary,
  we can replace $\ContinuousMap$ by a new map that agrees with
  $\ContinuousMap$ outside a neighborhood of the polyhedra and disks and still
  has those as the preimage of the $0$-block and so that the new map has the
  standard form on the boundary of these polyhedra.
  Now a map between cones is homotopic relative the boundary to the cone of the
  boundary map. Hence we can find a homotopic map
  satisfying~(\ref{itm:TransverseMembrane})-(\ref{itm:0Blocks}).
\end{proof}

\begin{proof}[Proof of Theorem~\ref{thm:NormalForm}]
  We can assume that $\ContinuousMap$ is in the form described in
  Lemma~\ref{lem:HomotopicCellwiseCovering}.
  Let us construct a generalized $2$-complex $\mathcal{S}$ that is homeomorphic
  to $\Surface$.
  Let $\CellsOfDim{\mathcal{S}}{i}$ denote the set of connected components of
  all preimages of open $i$-cells of $\TopologicalSpace$.
  Let
  $
    \CellType
    \colon
    \CellsOfDim{\mathcal{S}}{2}
    \to
    \ManifoldModelsOfDimension{2}
  $
  denote the map that sends such a connected component to the homeomorphism
  type of the preimage of the plasma of the open cell in that connected
  component.
  Note that the boundary of a connected component of the preimage of an open
  $i$-cell , necessarily lies in the preimage of the
  $\left(i-1\right)$-skeleton.
  Since $\ContinuousMap$ is of the form described in
  Lemma~\ref{lem:HomotopicCellwiseCovering}, the inner of the
  preimage of the plasma is homeomorphic to the preimage of the inner of the
  corresponding cell.
  This homeomorphism can be extended so that the boundary of the preimage of
  the plasma maps to the boundary of the preimage of the corresponding open
  cell.
  Define the gluing maps of $\mathcal{S}$ accordingly.
  Then it is evident that there is a homeomorphism
  $\Phi \GeometricRealization{\mathcal{S}}\to \Surface$ and that
  $\ContinuousMap\circ \Phi$ is a cellwise covering without folds.
\end{proof}

As an application of Theorem~\ref{thm:NormalForm} we prove the following
interesting property of the attainable set:
\begin{proposition}%
  \label{prp:NonOrientableCellsIrrelevant}
  Given a generalized $2$-complex $\GeneralizedComplex$,
  let $\GeneralizedComplex_{O}$ denote the subcomplex
  given by the orientable cells.
  Then the inclusion
  $
    \iota_O
    \colon
    \GeneralizedComplex_{O}
    \to
    \GeneralizedComplex
  $
  induces an isomorphism
  $
    \HomologyOfSpaceObject{\GeneralizedComplex_{O}}{2}
    \to
    \HomologyOfSpaceObject{\GeneralizedComplex}{2}
  $
  and for all
  $
    \HomologyClass
    \in
    \HomologyOfSpaceObject{\GeneralizedComplex_{O}}{2}
  $
  we have
  $
    \AttainableSet{\TopologicalSpace}{\alpha}
    =
    \AttainableSet{\TopologicalSpace}{
      \apply{
        \HomologyOfSpaceMorphism{\iota_{O}}{2}
      }
      {\HomologyClass}
    }
  $.
\end{proposition}
\begin{proof}
  By Lemma~\ref{lem:HomologyGeneralized2Complex}, $\iota_O$ induces an
  isomorphism on second homology.

  Given a non-orientable $\Cell\in\ManifoldModelsOfDimension{2}$ we write
  $
    \pi_{\Cell}
    \colon
    \hat{\Cell}
    \to
    \Cell
  $
  for its orientation cover.
  We write $\hat{\GeneralizedComplex}$ for the generalized $2$-complex,
  where the cell type of every non-orientable $2$-cell is replaced by its
  orientation cover
  and the gluing maps are given by the composition of the covering projection
  and the corresponding gluing map of $X$.
  Note that for every non-orientable $\Cell \in \CellsOfDim{X}{2}$, we have
  $
    \apply{\HomologyOfSpaceMorphism{\pi_{\Cell}}{1}}
      {\FundamentalClass{
        \partial
        \apply{\CellType}{\hat{\Cell}}
        }
      }
    =0
    \in
    \HomologyOfSpaceObject{
      \partial
      \apply{\CellType}{\Cell}
    }{1}
  $.
  By Corollary~\ref{crl:CellwiseCovering} every representative of a
  homology class has a smaller representative which is a cellwise
  covering and therefore lifts to
  $\GeometricRealization{\hat{\GeneralizedComplex}}$.
  Let us denote such a representative by
  $
    \ContinuousMap
    \colon
    \Surface
    \to
    \GeometricRealization{\GeneralizedComplex}
  $
  and a lift by $\hat{\ContinuousMap}$.

  Consider a non-orientable generalized $2$-cell $\Cell$. There are two
  cases:
  Either the preimage of $\IncludedPlasma{\Cell}$ is empty, in which case there
  is nothing to
  show, or
  $
    \at{
      \hat{\ContinuousMap}
    }{
      \apply{\ContinuousMap^{-1}}{\IncludedPlasma{\Cell}}
    }
  $ is a degree $d>0$ cover.

  So let us assume that we are in the second case.
  Since $\Surface$ is orientable and
  $\apply{\CellType}{\hat{\Cell}}$ is connected, we know that the fundamental
  class of $\apply{\ContinuousMap^{-1}}{\IncludedPlasma{\Cell}}$
  gets mapped to
  $
    d
    \FundamentalClass{
      \IncludedPlasma{\hat{\Cell}}
    }
  $.

  We know that
  $
    \FundamentalClass{
      \partial
      \apply{\ContinuousMap^{-1}}
        {\IncludedPlasma{\Cell}}
    }
  $
  maps to
  $
    d
    \FundamentalClass{
      \partial
      \IncludedPlasma{\hat{x}}
    }
  $
  and therefore
  $
    \FundamentalClass{
      \partial
      \apply{\ContinuousMap^{-1}}{\IncludedPlasma{\Cell}}
    }
  $
  maps to $0$ in
  $
    \HomologyOfSpaceObject{
      \partial \IncludedPlasma{\Cell}}
    {1}
  $.
  Let us denote a connected component of $\partial \IncludedPlasma{\Cell}$ by
  $c$, and the intersection of
  $
    \apply{\ContinuousMap^{-1}}{c}
  $
  and a connected component of
  $\apply{\ContinuousMap^{-1}}{\IncludedPlasma{\Cell}}$ by
  $\tilde{c}$.
  Then $\at{\ContinuousMap}{\tilde{c}}$ extends to a sphere with
  $
    \Cardinality{
      \PathComponentsOfSpace{\tilde{c}}
    }
  $
  holes such that the image of this sphere is contained in $c$.
  Indeed the image of the fundamental class of $\tilde{c}$ in
  $
  \HomologyOfSpaceObject{
    \partial \IncludedPlasma{\Cell}}
  {1}
  $
  being zero, implies the existence of such an extension.

  Therefore we can replace
  $
    \ContinuousMap
    \colon
    \Surface
    \to
    \GeometricRealization{X}
  $
  by deleting the preimages of all non-orientable cells and replacing them by
  the previously constructed spheres with holes, one for every $c$. Since a
  disjoint union of spheres with holes has a lower genus and a lower
  $\ChiMinus$ than all connected surfaces with the same number of boundary
  components, we conclude that this procedure reduces $\Genus$ and $\ChiMinus$.
  The resulting map is evidently homotopic to a map that avoids the
  non-orientable cells.
\end{proof}
\begin{remark}
  While all the attainable sets of $\GeneralizedComplex_{O}$ and
  $\GeneralizedComplex$ agree, it is not true in
  general that $\iota_{O}$ is minimizer surjective.
\end{remark}

\section{Effective Computability}
\label{scn:WeightFunctions}
To talk about the computability of invariants of generalized 2-complexes,
it has to be explained how a generalized 2-complex can be encoded
combinatorially in a machine-understandable way.
There are several ways of fixing such an encoding.
We will briefly describe one option here:

We require each cell set to be given as an explicit set of labels
(natural numbers, or words in the usual alphabet for readability).
The gluing maps of 1-cells are given by denoting for each 1-cell the
labels of its start- and endpoint.
The gluing map of each boundary component of a 2-cell
is given by a non-empty sequence of 1-cells and their formal inverses,
or a 0-cell.
Every combinatorial generalized 2-complex can be encoded this way.
For example, the complex from Example~\ref{exm:Handlebody} has the following
encoding:

\begin{verbatim}
  ( ( [0-cells]
      p [label]
    ),
    ( [1-cells]
      ( a [label], (p [start], p [end]) ),
      ( b [label], (p [start], p [end]) ),
      ( c [label], (p [start], p [end]) )
    ),
    ( [2-cells]
      ( X [label],
        ( 1 [number of boundary components],
          0 [genus],
          + [orientable: yes]
        ),
        ( a, b, a^{-1}, b^{-1} ) [gluing map]
      ),
      ( Y [label],
        ( 1 [number of boundary components],
          0 [genus],
          + [orientable: yes]
        ),
        ( a, c, a^{-1}, c^{-1} ) [gluing map]
      ),
      ( Z [label],
        ( 1 [number of boundary components],
          0 [genus],
          + [orientable: yes]
        ),
        ( a, b, c, a^{-1}, c^{-1}, b^{-1} ) [gluing map]
      )
    )
  )
\end{verbatim}

In this section we will establish the computability of the attainable set for a
large class of spaces.
This class of spaces is characterized by the existence of a "weight function".

Let $\GeneralizedComplex$ be a generalized combinatorial 2-complex.
\begin{definition}[Corners]
  The connected components of the subset of
  $
    \bigsqcup_{\Cell \in \CellsOfDim{\GeneralizedComplex}{2}}
    \apply{\CellType}{\Cell}
  $
  that gets mapped to 0-cells are either points or circles.
  Let the \introduce{set of corners of 2-cells}
  $\apply{C_{2}}{\GeneralizedComplex}$
  be the set of those connected components which are points.

  Let $r_{2,2}$ and $r_{2,0}$ be the canonical maps
  \[
    \left(r_{2,0}, r_{2,2}\right)
    \colon
    \apply{C_{2}}{\GeneralizedComplex}
    \rightarrow
    \CellsOfDim{\GeneralizedComplex}{0}
    \times
    \CellsOfDim{\GeneralizedComplex}{2}.
  \]

  Let the \introduce{set of corners of 1-cells} $\apply{C_{1}}{\GeneralizedComplex}$
  denote the set of points in
  $
    \bigsqcup_{\Cell \in \CellsOfDim{X}{1}}
    \apply{\CellType}{\Cell}
  $
  that get mapped to 0-cells by the gluing map.
  Let $r_{1,1}$ and $r_{1,0}$ be the canonical maps
  \[
    \left(r_{1,0}, r_{1,1}\right)
    \colon
    \apply{C_{1}}{\GeneralizedComplex}
    \rightarrow
    \CellsOfDim{\GeneralizedComplex}{0}
    \times
    \CellsOfDim{\GeneralizedComplex}{1}.
  \]
  Let
  $
    e^{+}, e^{-} \colon
    \apply{C_{2}}{\GeneralizedComplex}
    \rightarrow
    \apply{C_{1}}{\GeneralizedComplex}
  $
  be the maps with
  $
    r_{1,0} \circ e^{+}
    = r_{2,0}
  $,
  $
    r_{1,0} \circ e^{-}
    = r_{2,0}
  $,
  which send a corner to the corresponding corner
  of the image of the boundary segment to its right (left)
  under the gluing map (right and left are given by an orientation on
  the boundary of the 2-cell, the choice of the orientation is irrelevant).
\end{definition}

\begin{definition}[Link Sequence]
  For a zero-cell $\Cell$,
  a \introduce{link sequence around $\Cell$} is a sequence
  $
    (c_{1}, o_{1}), \ldots, (c_{n}, o_{n})
    \in
    \apply{r^{-1}_{2,0}}{\Cell}\times \{+,-\}
  $
  such that
  $
    \apply{e^{o_{i}}}{c_{i}}
    =
    \apply{e^{-o_{i+1}}}{c_{i+1}}
  $.
  It is said to be \introduce{without folds} if no corner is followed by
  its inverse.
\end{definition}

\begin{definition}[Weight Function]%
\label{dfn:WeightFunction}
  A \textbf{weight function} for $\GeneralizedComplex$ is a map
  $
    \WeightFunction
    \colon
    \apply{C_{2}}{\GeneralizedComplex}
    \rightarrow
    \Rationals
  $
  such that
  \begin{enumerate}[(1)]
    \item\label{itm:WeightedEuler}
      the \introduce{weighted Euler characteristic}
      $\apply{\WeightedEulerCharacteristic}{\Cell}$
      of $\Cell$ is negative,
      where $\WeightedEulerCharacteristic$ is defined as
      \[
        \apply{\WeightedEulerCharacteristic}{\Cell} =
        \apply{\EulerCharacteristic}{\apply{\CellType}{\Cell}}
          - 0.5 \times \apply{\NumberOfBoundarySegments}{\Cell}
          + \apply{\NumberOfBoundaryCircles}{\Cell}
          + \sum_{c \in \apply{r^{-1}_{2,2}}{\Cell}}
              \apply{\WeightFunction}{\Cell}
      \]

      Here the number of boundary segments
      $\apply{\NumberOfBoundarySegments}{\Cell}$ is defined as the
      number of connected components of the preimage of
      $\SkeletonOfDim{\GeneralizedComplex}{0}$ under
      $
        \GluingMap{\Cell}
        \colon
        \partial
        \apply{\CellType}{\Cell}
        \rightarrow
        \SkeletonOfDim{X}{1}
      $
      which are points,
      and the number of boundary circles
      $\apply{\NumberOfBoundaryCircles}{\Cell}$
      as the number of those connected components which are circles.
    \item\label{itm:LinkSequence}
      For each non-empty link sequence
      $
        \left(
          \left(c_{1},o_{1}\right),
          \ldots,
          \left(c_{n},o_{n}\right)
        \right)
      $
      without folds we have:
      \[
        \apply{\WeightFunction}{c_{1}}
        + \cdots
        + \apply{\WeightFunction}{c_{n}}
        \geq 1
      \]
  \end{enumerate}
\end{definition}

\begin{remark}
  Given a $2$-dimensional $M_{-1}$ polyhedral complex (See Section~I.7 of
  \cite{BridsonHaefligerMetric}), then this is a $CAT(-1)$-space (See
  Section~II.1 of \cite{BridsonHaefligerMetric}) if and only if the angles of
  the incoming edges at every vertex define a weight function (Theorem~5.5 in
  Section~II.5 of \cite{BridsonHaefligerMetric}).

  Conversely every $2$-complex that has a weight function with only
  positives weights can be turned into a $CAT(-1)$-space by choosing a metric
  of constant curvature $-1$ on each $2$-cell.
\end{remark}

\begin{theorem}[Effective Computability]%
\label{thm:EffectiveComputability}
  There is an algorithm, taking as input
  an (encoding of a) connected generalized 2-complex $\GeneralizedComplex$,
  a weight function $w$ on $\GeneralizedComplex$,
  and a homology class $\HomologyClass$ in
  $\HomologyOfSpaceObject{\GeneralizedComplex}{2}$,
  which computes
  $\AttainableSet{\GeneralizedComplex}{\HomologyClass}$.
\end{theorem}

\begin{proof}[Proof of Theorem~\ref{thm:EffectiveComputability}]
  We prove that all minima in the product order
  $\AttainableSet{\GeneralizedComplex}{\HomologyClass}$
  have representatives
  $
    \ContinuousMap
    \colon
    \Surface
    \rightarrow
    \GeneralizedComplex
  $
  which are cellwise covers without folds
  and such that the sum of all degrees over all 2-cells of $\ContinuousMap$
  is bounded by a computable constant $\apply{C}{\HomologyClass}$
  depending only on
  $\GeneralizedComplex$, $\HomologyClass$ and $\WeightFunction$.
  Then there are only finitely many of such
  $\left(\Surface,\ContinuousMap\right)$, and they are enumerable.
  Hence one can compute the attainable set
  by computing $\Genus$ and $\ChiMinus$ of these representatives,
  and then saturate (Proposition~\ref{prp:AttainableSetBounds},
  Lemma~\ref{lem:AttainableSetSaturation})
  to obtain $\AttainableSet{\GeneralizedComplex}{\HomologyClass}$.

  Let $\left(\Surface,\ContinuousMap\right)$
  be any representative of $\HomologyClass$.
  Such a representative can be constructed combinatorially
  by taking $\apply{n}{\Cell}$ many copies of every cell $\Cell$ of
  $\GeneralizedComplex$ and gluing them arbitrarily,
  where $\apply{n}{\Cell}$ is defined by
  $
    \HomologyOfSpaceObject{\GeometricRealization{\GeneralizedComplex}}{2}
    \ni
    \HomologyClass
    \mapsto
    \sum_{\Cell \in \CellsOfDim{\GeneralizedComplex}{2}}
      \apply{n}{\Cell}
      \apply{
        \HomologyOfSpaceMorphism{\CellInclusion{\Cell}}{2}
      }{
        \FundamentalClass{%
          \apply{\CellType}{\Cell},
          \partial\apply{\CellType}{\Cell}
        }
      }
    \in
    \HomologyOfSpacePairObject{
      \GeometricRealization{\GeneralizedComplex}
    }{
      \SkeletonOfDim{\GeneralizedComplex}{1}
    }{2}
  $.
  The first step in the computation of
  $\AttainableSet{\GeneralizedComplex}{\HomologyClass}$
  is the construction of such an arbitrary representative.
  In particular, we obtain the number $B=2\apply{\Genus}{\Surface}$,
  which we need as an upper bound in the next step.
  This number may depend on the choice of the gluing.

  Let $\left(\Surface',\ContinuousMap'\right)$ be a representative of a
  minimum of $\AttainableSet{\GeneralizedComplex}{\HomologyClass}$.
  We want to prove that $\apply{\ChiMinus}{\Surface'} \leq B$.
  To prove this, note that the minimality implies
  $
    \left(
      \apply{\ChiMinus}{\Surface'}-2,
      0.5\ \apply{\ChiMinus}{\Surface'}
        + \apply{\NumberOfNonsphericalComponents}{\Surface'}
    \right)
    \not\in
    \AttainableSet{\GeneralizedComplex}{\HomologyClass}
  $.
  Hence
  $
    0.5\ \apply{\ChiMinus}{\Surface'}
    \leq
    0.5\ \apply{\ChiMinus}{\Surface'}
      + \apply{\NumberOfNonsphericalComponents}{\Surface'}
    \leq
    \apply{\Genus_{c}}{\HomologyClass}
    \leq
    \apply{\Genus}{\Surface}
  $.

  By Corollary~\ref{crl:CellwiseCovering}, every minimum of
  $\AttainableSet{\GeneralizedComplex}{\HomologyClass}$
  has a representative which is a cellwise covering without folds.
  Let $\left(\Surface',\ContinuousMap'\right)$ be such a representative,
  and consider $\Surface'$ to be a generalized 2-complex
  with the induced cell structure i.e. we identify $\Surface'$ with the
  corresponding $\mathcal{S}$ in Definition~\ref{dfn:CellwiseCovering}.
  Then, denoting by $d_{\Cell}$ the covering degree
  of $\ContinuousMap'$ over the cell $\Cell$,
  we have:
  \begin{alignat}{8}
    \nonumber
    - \apply{\ChiMinus}{\Surface'}
    \leq
    \apply{\EulerCharacteristic}{\Surface'}
    =
    && \sum_{\Cell \in \CellsOfDim{\Surface'}{2}} \phantom{d_{\Cell}} \Bigg(
    &    \apply{\EulerCharacteristic}{\apply{\CellType}{\Cell}}
       \Bigg)
    && - \left| \CellsOfDim{\Surface'}{1} \right|
    && + \left| \CellsOfDim{\Surface'}{0} \right|
    \\
    \label{eqn:BoundarySegments}
    =
    && \sum_{\Cell \in \CellsOfDim{\GeneralizedComplex}{2}} d_{\Cell} \Bigg(
    &    \apply{\EulerCharacteristic}{\apply{\CellType}{\Cell}}
    &&  - 0.5\ \apply{\NumberOfBoundarySegments}{\Cell}
       \Bigg)
    && + \left| \CellsOfDim{\Surface'}{0} \right|
    \\
    \label{eqn:ApplyLinkInequality}
    \leq
    && \sum_{\Cell \in \CellsOfDim{\GeneralizedComplex}{2}} d_{\Cell} \Bigg(
    &    \apply{\EulerCharacteristic}{\apply{\CellType}{\Cell}}
    && - 0.5\ \apply{\NumberOfBoundarySegments}{\Cell}
    && + \sum_{c \in \apply{r^{-1}_{2,2}}{\Cell}} \apply{w}{c}
       + \apply{\NumberOfBoundaryCircles}{\Cell}
       \Bigg)
    \\
    \nonumber
    =
    && \sum_{\Cell \in \CellsOfDim{\GeneralizedComplex}{2}} d_{\Cell} \Bigg(
    &    \apply{\WeightedEulerCharacteristic}{\Cell}
       \Bigg)
    &&
    &&
    \\
    \nonumber
    \leq
    &&   \left(
           \sum_{\Cell \in \CellsOfDim{\GeneralizedComplex}{2}} d_{\Cell}
         \right)
    &    \max_{c}{\apply{\WeightedEulerCharacteristic}{c}}
    &&
    &&
  \end{alignat}
  Equation~(\ref{eqn:BoundarySegments}) holds because every 1-cell of
  $\Surface'$ appears two times as a boundary segment.
  Inequality~(\ref{eqn:ApplyLinkInequality})
  holds because of the link condition (\ref{itm:LinkSequence})
  in Definition~\ref{dfn:WeightFunction}.
  Because of $\max_{c}{\apply{\WeightedEulerCharacteristic}{c}}<0$
  we obtain:
  \[
    \sum_{\Cell \in \CellsOfDim{\GeneralizedComplex}{2}} d_{\Cell}
    \leq -\apply{\ChiMinus}{\Surface'} /
    \max_{c}{\apply{\WeightedEulerCharacteristic}{c}}
    \leq -B / \max_{c}{\apply{\WeightedEulerCharacteristic}{c}}
  \]
\end{proof}

\begin{example}[Three Octagons, continued]
\label{exm:ThreeOctagonsContinued}
  We show how the attainable sets from
  Example~\ref{exm:ThreeOctagons}
  can be computed using the following weight function: Every black corner in
  Figure~\ref{fig:WeightsThreeOctagons} has weight $\frac{1}{8}$ and every
  white corner has weight $\frac{1}{2}$.
  With this weight function,
  the weighted Euler characteristic of each cell is $-2$.
  \begin{figure}[h]
    \def\svgwidth{0.7\textwidth}
    \input{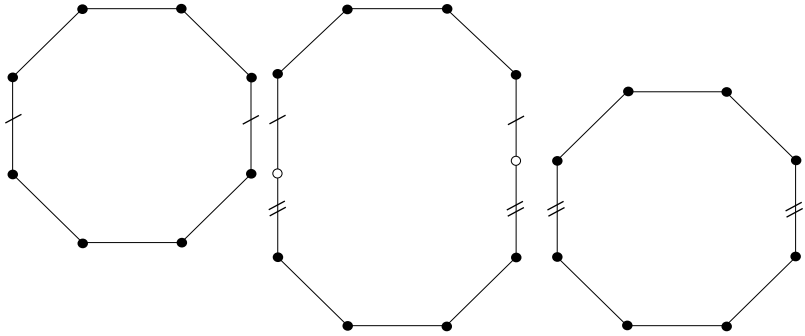}
    \caption{%
      Weights on three octagons, Example~\ref{exm:ThreeOctagonsContinued}%
    }
    \label{fig:WeightsThreeOctagons}
  \end{figure}

  In a first step, we construct a connected representative
  $\Surface_{n}^{\textrm{c}}$ for $n\HomologyClass$ with
  $%
    (
      \apply{\ChiMinus}{\Surface_{n}^{\textrm{c}}},%
      \apply{\Genus}{\Surface_{n}^{\textrm{c}}}%
    )%
    =(6n+2,3n+2)%
  $.
  We do this by gluing $n$ copies of each cell in a chain, and then
  gluing these chains together, as depicted in
  Figure~\ref{fig:ThreeOctagonsConnectedRepresentative}.
  This complex has $(4n-2)$ $0$-cells, $13n$ $1$-cells and $3n$ $2$-cells,
  hence the formulas for $\apply{\ChiMinus}{\Surface_{n}^{\textrm{c}}}$ and
  $\apply{\Genus}{\Surface_{n}^{\textrm{c}}}$.
  \begin{figure}[h]%
    \def\svgwidth{\textwidth}
    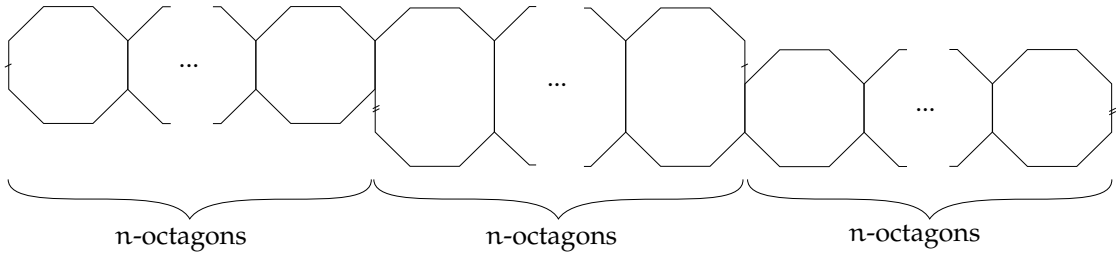
    \caption{%
      Connected representative, Example~\ref{exm:ThreeOctagonsContinued}%
    }
    \label{fig:ThreeOctagonsConnectedRepresentative}
  \end{figure}

  Hence, for any $(\ChiMinus,\Genus)$-minimizer $\Surface_{n}$ representing
  $n\HomologyClass$,
  $%
    6n
    \leq
    \apply{\ChiMinus}{\Surface_{n}}%
    \leq%
    \apply{\ChiMinus}{\Surface_{n}^{\textrm{c}}}%
    \leq 6n+2%
  $.
  Each representative which is a cellwise cover without folds
  covers each cell exactly $n$ times
  (by the weighted Euler characteristic bound:
  a priori one cell could be covered $(n+1)$ times, but this would yield a
  homological degree of $(n+1)$ or $(n-1)$, not $n$).

  As a second step, we construct a non-connected representative
  $\Surface_{n}^{\textrm{nc}}$ for the class $n\HomologyClass$ such that
  $%
    (%
      \apply{\ChiMinus}{\Surface_{n}^{\textrm{nc}}},%
      \apply{\Genus}{\Surface_{n}^{\textrm{nc}}}%
    )%
    = (6n,3n+3)%
  $
  by taking a regular connected $n$-fold cover of each cell.
  This representative has three connected components of genus $(n+1)$,
  hence the formulas.

  Now all that is left to show is that there is no surface $\Surface_{n}$
  representing the class $n\HomologyClass$ with
  $%
    (%
      \apply{\ChiMinus}{\Surface_{n}^{\textrm{nc}}},%
      \apply{\Genus}{\Surface_{n}^{\textrm{nc}}}%
    )%
    =(6n,3n+2)%
  $.
  For such a surface, the estimate for $\ChiMinus$ given by the weighted Euler
  characteristic would be sharp. Hence, the sum over the weights for each link
  used in the cover has to be $1$. One can see that this is only possible
  if each corner of weight $\frac{1}{2}$ is glued to another corner of weight
  $\frac{1}{2}$. But then the representative is a disjoint union of surfaces
  each mapping to single cells, hence it has at least $3$ components,
  contradicting
  $%
    \apply{\NumberOfNonsphericalComponents}{\Surface_{n}}%
    =%
    \apply{\Genus}{\Surface_{n}} -%
    \frac{1}{2}\apply{\ChiMinus}{\Surface_{n}} = 2%
  $

  In conclusion, the only minima of the attainable set are $(6n+2, 3n+2)$ and
  $(6n,3n+3)$.
\end{example}

\section{Undecidability}
\label{scn:Undecidability}
We need the following construction introduced in~\cite{GordonUndecidability}:
Given a finite presentation
\[
  \Pi = \left( q_{1}, \ldots, q_{m}; r_{1}, \ldots r_{n} \right)
\]
of a group $\Group_{\Pi}$ and $w$ a word in $q_{1}, \ldots, q_{m}$,
let $\Pi_{w}$ denote the presentation given by $\Pi$ together with
extra generators $a,\alpha,b,\beta$ and additional relations:
\begin{enumerate}[(i)]
  \item $a \alpha a^{-1} = b^{2}$
  \item $\alpha a \alpha^{-1} = b \beta b^{-1}$
  \item
    $
      a^{2i} q_{i} \alpha^{2i}
      =
      \beta^{2\left(i+1\right)} b \beta^{-2\left(i+1\right)}
    $,
    $1 \leq i \leq m$
  \item $\left[w,a\right] = \beta^{2} b \beta^{-2}$
  \item $[w,\alpha] = \beta b \beta b^{-1} \beta^{-1}$
\end{enumerate}
The group presented by $\Pi_{w}$ will be denoted as $\Group_{\Pi,w}$.
The following properties of $\Group_{\Pi,w}$ are
proven in the proof of Theorem~$3$ in~\cite{GordonUndecidability}:
\begin{itemize}
  \item
    If $w$ represents the unit element in $\Group_{\Pi}$,
    $\Group_{\Pi,w}$ is trivial.
  \item
    Otherwise, $a$ and $b$ have infinite order in $\Group_{\Pi,w}$.
  \item
    In any case, $a$ and $b$ are commutators in $\Group_{\Pi,w}$,
    namely
    $
      a
      =
      \left[
        \alpha^{-1} \beta^{-1} w \beta \alpha,
        \alpha^{-1} \beta^{-1} \alpha \beta \alpha
      \right]
    $,
    which follows from the additional relations~(i) and~(v),
    and
    $
      b
      =
      \left[
        \beta^{-2} w \beta^{2},
        \beta^{-2} a \beta^{2}
      \right]
    $,
    which follows from the additional relation (iv).
\end{itemize}
\begin{proof}[Proof of Theorem~\ref{thm:MinimalGenusIsUndecidable}]
  First we will prove this for $i = \Genus$.
  Let $\Pi = \left(Q;R\right)$ be a finite presentation of a group $G_{\Pi}$.
  Let $w$ be a word in the generators $Q$.

  We define for every $\left(\Pi,w\right)$ as above a 2-complex
  $\GeneralizedComplex_{\Pi,w}$ and a class
  $
    \HomologyClass_{\Pi,w}
    \in
    \HomologyOfSpaceObject%
      {\GeometricRealization{\GeneralizedComplex_{\Pi,e}}}%
      {2}
  $
  in the following way:
  $\GeneralizedComplex_{\Pi,w}$ is the union of the presentation
  complex of $\Pi_{w}$ and a cell of type $\Sigma_{B,2}$
  glued to it via $a$ and $b$.
  The class $\HomologyClass_{\Pi,w}$ is the fundamental class of
  $\Sigma_{B,2}$ plus the sum of the classes of the two commutators above.

  Now this class is represented by a surface of genus bounded by $B$ if and
  only if $w$ was trivial.
  Hence, if there were an algorithm that computes minimal genus,
  there would be an algorithm solving the word problem,
  which is known to be undecidable (See~\cite{NovikovWordProblem}).

  Every minimizing representative for the problem above is automatically
  essentially connected.
  Because for essentially connected surfaces $\ChiMinus$ and $\Genus$ are
  related, both problems are undecidable.
\end{proof}

\bibliography{sources}
\bibliographystyle{plain}
\end{document}